\documentclass{amsart}
\usepackage{fullpage}
\usepackage{amsmath,amssymb}
\usepackage[dvipsnames]{xcolor}
\usepackage{tikz}
\usepackage{float}
\usepackage[colorlinks=true, linkcolor=blue, citecolor=green, urlcolor=cyan]{hyperref}

\newtheorem*{theor*}{Theorem}
\newtheorem*{prop*}{Proposition}

\newtheorem{theor}{Theorem}[section]

\newtheorem{lemma}[theor]{Lemma}
\newtheorem{prop}[theor]{Proposition}
\newtheorem{cor}[theor]{Corollary}
\newtheorem{problem}[theor]{Problem}
\newtheorem{conjecture}[theor]{Conjecture}
\newtheorem{definition}[theor]{Definition}
\newtheorem{remark}[theor]{Remark}

\newcommand{\N}{\mathbb{N}}
\newcommand{\R}{\mathbb{R}}

\newcommand{\Exp}{\mathbb{E}}

\newcommand{\smin}{s_{\min}}

\newcommand{\good}[1]{#1_{\mathrm{good}}}
\newcommand{\flowdown}{\vcenter{\hbox{\tikz{\draw[->,double distance=1.1pt,line width=0.35pt] (0,0.58) -- (0,-0.58);}}}}
\newcommand{\keyterm}[1]{\textcolor{PineGreen}{\emph{#1}}}
\newcommand{\lsfci}{\hyperref[def:local-sparse-four-cycle-incidence]{\keyterm{local sparse four-cycle incidence}}}
\newcommand{\lsdoubleoverlap}{\hyperref[ass:local-sparse-double-overlap]{\keyterm{local sparse double-overlap condition}}}

\newcommand{\doubleoverlappartners}{\hyperref[def:double-overlap-partner]{\keyterm{double-overlap partners}}}
\newcommand{\fcp}{\hyperref[def:local-sparse-four-cycle-incidence]{\keyterm{four-cycle partner}}}
\newcommand{\overlapparameter}{\hyperref[def:overlap-parameter]{\keyterm{overlap parameter}}}
\newcommand{\reservoir}{\hyperref[def:c4-witness-block]{\keyterm{reservoir}}}
\newcommand{\witnessblock}{\hyperref[def:c4-witness-block]{\keyterm{witness block}}}
\newcommand{\light}{\hyperref[def:heavy-neighbor-light-root]{\keyterm{light}}}
\newcommand{\mrt}[1]{\hyperref[def:rooted-matching-tree]{\keyterm{$#1$-matching rooted tree}}}
\newcommand{\wcmrt}[1]{\hyperref[def:weight-compatible-rooted-matching-tree]{\keyterm{weight-compatible $#1$-matching rooted tree}}}
\newcommand{\ff}[2]{\hyperref[eq:falling-factorial]{\textcolor{black}{(#1)_{#2}}}}
\newcommand{\tref}[1]{\ifmmode\text{\tiny\ref{#1}}\else{\tiny\ref{#1}}\fi}

\title{Well-invertible column subsets of sparse matrices
are rare}
\author{Han Huang}
\author{Mark Rudelson}
\author{Konstantin Tikhomirov}

\address{Department of Mathematics, University of Missouri, Columbia}\email{hhuang@missouri.edu} 

\address{Department of Mathematics, University of Michigan}\email{rudelson@umich.edu} 

\address{Department of Mathematical Sciences, Carnegie Mellon University}\email{ktikhomi@andrew.cmu.edu}

\begin{document}

\begin{abstract}
A random $n\times k$ matrix $S$ is an \emph{$(r,\alpha)$-oblivious subspace injection} (OSI) if 
$\Exp\|S^\top x\|_2^2=\|x\|_2^2$ for every $x\in\R^n$, and for every fixed $r$-dimensional subspace $V\subset\R^n$, with probability close to one,
one has $\alpha\|x\|_2^2\le\|S^\top x\|_2^2$ for all $x\in V$.
In this work, 
we show that
in the regime $r=\Omega(k)$ and $\alpha=\Omega(1)$,
and under a mild additional structural assumption,
no constant-row-sparsity matrix $S$ is OSI,
thereby answering, in a strong form, a question
raised by Cama\~no, Epperly, Meyer, and Tropp.

We show that the failure of the OSI property for sparse random matrices stems from a general deterministic phenomenon, thereby reducing a probabilistic problem to a non-probabilistic one. This phenomenon is related to the restricted invertibility principle introduced in the seminal work of Bourgain--Tzafriri.
Let $(n_k)_{k\in\N}$ be a sequence of integers
satisfying $\frac{n_k}{k}\to\infty$.
For each $k$, let $S^{(k)}$ be
a $n_k\times k$ non-random matrix with $O(1)$ nonzero entries per row,
whose nonzero entries have average magnitude $O(1)$,
and such that the total number of pairs of rows with
supports overlapping at two or more indices
is $o({n_k}^2/k)$.
We prove that for every constant
$\varepsilon>0$, as $k\to\infty$, the overwhelming majority of 
$k\times \lfloor\varepsilon k\rfloor$ submatrices of $(S^{(k)})^\top$
have the smallest singular value $o(1)$.
Thus, the well-invertible submatrices whose existence is guaranteed by the Bourgain--Tzafriri theorem are rare. 
The proof
is itself based on probabilistic tools.

\end{abstract}

\maketitle
\tableofcontents

\section{Introduction}
\label{sec:introduction}

\subsection{Motivation: oblivious subspace injections}

Randomized sketching has become a fundamental technique in modern numerical linear algebra for lowering the computational cost of large-scale problems; see, for example, \cite{Mahoney2011,
Woodruff2014,MartinssonTropp2020}.
With randomized sketching, a high-dimensional dataset
is replaced
by a much smaller random projection obtained by applying a dimension
reduction via a {\it sketch matrix}. With
an appropriately constructed sketch matrix,
the compressed representation retains the key geometric features of the original data, making it possible to compute fast approximate solutions with rigorous guarantees.
A by-now classical notion in randomized sketching is that
of {\it oblivious subspace embedding} (OSE). A random $n\times k$ sketch matrix $S$ is an oblivious subspace embedding with parameters $r$ and $\varepsilon$,
if for every non-random subspace $V$ of $\R^n$ of dimension $r$, the linear operator $S^\top$ restricted to $V$ acts as a $(1+\varepsilon)$--isometry.
Sparse Euclidean subspace embeddings were developed, in particular, in \cite{ClarksonWoodruff2013,MengMahoney2013,NelsonNguyenOSNAP,NelsonNguyenOSELower,Cohen2016,
ChenakkodEtAl2024,ChenakkodEtAl2025,ChenakkodEtAl2026}; we refer to bibliography overview in \cite{ChenakkodEtAl2026} for an account
of this very active line of research.

Rather than two-sided bounds on distortion considered in the OSE setting, one can restrict attention to only quantifying operator's injectivity.
The motivation for our note comes from this framework, systematically explored by Camaño, Epperly, Meyer, and Tropp under the name of {\it oblivious subspace injection (OSI)} \cite{CamanoEtAl2025}. 
The work \cite{CamanoEtAl2025} shows that several randomized linear algebra algorithms succeed under only the OSI hypothesis,
as long as one tolerates constant multiplicative approximation error
(at the same time, OSI has certain limitations which are not discussed here; see \cite{TownsendWang2026}). 
\begin{definition}[{Oblivious subspace injection property;
\cite[Definition~1.2]{CamanoEtAl2025}}]
\label{def:osi}
Fix integers $r \le k \le n$ and a parameter $\alpha \in (0,1]$. A random  $n\times k$ matrix $S$
is \emph{$(r,\alpha)$-OSI} if it satisfies the following two conditions:
\begin{itemize}
    \item[(1)] \emph{Isotropy:} for every vector $x \in \R^n$,
    \[
        \Exp \|S^\top x\|_2^2 = \|x\|_2^2.
    \]
    \item[(2)] \emph{Injectivity:} for every fixed $r$-dimensional subspace $V \subset \R^n$, with probability at least $9/10$,
    \[
        \alpha \|x\|_2^2 \le \|S^\top x\|_2^2
        \qquad \forall x \in V.
    \]
\end{itemize}
The probability $9/10$ in the second condition
can be replaced with another number close to one.
\end{definition}

Given a design matrix $A\in\R^{n\times r}$ and a response matrix
$B\in\R^{n\times m}$, multivariate regression problem asks to solve
$\min_{X\in\R^{r\times m}}\|AX-B\|_{\mathrm F}^2$.
Instead of solving this problem directly, one can draw a random $n\times k$ sketch matrix $S$ and
compute
$
    \widetilde X
    \in
    \operatorname*{arg\,min}_{X\in\R^{r\times m}}
    \|S^\top(AX-B)\|_{\mathrm F}^2
$
by taking
$\widetilde X:=(S^\top A)^\dagger S^\top B.$
The authors of \cite{CamanoEtAl2025} show that if $S$ is a $(r,\alpha)$-OSI,
then with positive constant probability
\[
    \|A\widetilde X-B\|_{\mathrm F}^2
    \le
    \frac{C}{\alpha}
    \min_X\|AX-B\|_{\mathrm F}^2
\]
for a universal constant $C$ \cite[Theorem~2.7]{CamanoEtAl2025}.
Furthermore, if $S$ is sparse with $\zeta$ nonzeros per row, the sketches $S^\top A$ and
$S^\top B$ can be formed in $O(\zeta n(r+m))$ arithmetic operations for dense inputs,
instead of the $O(k n(r+m))$ operations required by a dense $n\times k$
sketch matrix.

Another example is the single-pass Nystr\"om approximation for a positive
semidefinite matrix. Given a PSD $A\in\R^{n\times n}$ and a sketch matrix
$S\in\R^{n\times k}$, one forms the sketch
$Y=A S$ and returns the rank-$k$ positive semidefinite approximation
$
    \widehat A
    =
    Y(S^\top Y)^\dagger Y^\top .
$
The authors of \cite{CamanoEtAl2025} prove that if $S$ is an
$(r,\alpha)$-OSI with $k\ge r$, then, with probability at least $9/10$,
\[
    \|A-\widehat A\|_*
    \le
    \frac{C}{\alpha}
    \|A-[[A]]_r\|_*
\]
for a universal constant $C$ \cite[Corollary~2.5]{CamanoEtAl2025},
where $\|\cdot\|_*$ denotes the nuclear norm, and
$[[A]]_r$ is the best rank $r$ approximation of $A$ in Frobenius norm.
For a sparse OSI with $\zeta$ nonzeros per row, forming the Nystr\"om sketch
$Y=A S$ for dense matrix $A$ costs $O(\zeta n^2)$.

In view of the above discussion, 
the optimal relation between OSI parameters --- sketch matrix sparsity and dimensions, as well as dimension of the oblivious subspace, and the magnitude of the injectivity parameter $\alpha$ ---
are of central importance. 
As an example of recent developments,
Tropp \cite{Tropp2026Comparison} proved that a $n\times k$ SparseStack map
$S$ with row sparsity
$
\zeta\ge 6\delta^{-1}\log(C\,r)$
and with
$k\ge 16\delta^{-2}r,$
is $(r,1-\delta)$-OSI.
Which sparsity level is sufficient for the OSI property is a compelling open problem. 
In \cite{CamanoEtAl2025}, the authors 
conjectured that constant sparsity 
{\it SparseStack} matrices (to be defined later)
are $(r,1/2)$-OSI for some embedding dimension $k=O(r)$.
As a much more general problem, in this paper we consider
\begin{problem}[Constant sparsity OSI matrices]\label{main problem}
Can an $n\times k$ random sketch matrix $S$
with constant row sparsity have
$(r,\alpha)$-OSI property for
constant $\alpha$ and with $k=O(r)$?
\end{problem}

The main result of this work shows that, under very broad assumptions including, in particular, SparseStack
model, constant sparsity sketch matrices {\it do not} possess OSI property for proportional oblivious subspace dimension and with constant injectivity parameter.
Moreover, we show that the injectivity property (2) in Definition \ref{def:osi} fails even if we restrict the test spaces $V \in \R^n$ to $r$-dimensional {\it coordinate} subspaces.

Problem \ref{main problem} is probabilistic. However, as we explain in the next section, the failure of the OSI property described above is due to a general phenomenon pertaining to deterministic matrices. This phenomenon called {\it restricted invertibility}  is a classical object of study in geometric functional analysis. Roughly speaking, it states that a matrix which is far away from the set of low-rank matrices has a well-invertible submatrix. 
This phenomenon will be discussed in more detail in the next section. Its relevance to Problem \ref{main problem} follows from our main result, which shows that, under mild conditions, a sparse matrix has very few such well-invertible submatrices.

\subsection{Well-invertible column subsets}

The seminal {\it restricted invertibility} theorem of Bourgain--Tzafriri
\cite{BourgainTzafriri1987,BourgainTzafriri1989,BourgainTzafriri1991} states that there are universal constants
$c,\gamma>0$ such that every $n\times n$ matrix $A$ with columns of Euclidean
norm one contains a column index set $\sigma\subset[n]$ with
$|\sigma|\ge c n/\|A\|^2$ and $
\smin(A_\sigma)\ge \gamma,$
where $\|A\|$ is the spectral norm of $A$, and $A_\sigma$ denotes the corresponding column submatrix. Vershynin
\cite{VershyninJohnsDecompositions} later extended this principle to
rectangular matrices, with the number of selected columns governed
by the stable rank $\|A\|_{\mathrm{HS}}^2/\|A\|^2$. The literature on
restricted invertibility is extensive and touches on various problems in local Banach space theory, convex geometry, and operator algebras; see, for example,
\cite{SpielmanSrivastava2012,Youssef2014,MarcusSpielmanSrivastava2015,MarcusSpielmanSrivastava2017,NaorYoussef2017}, and, in particular, recent work \cite{NaorYoussef2017}
for an in-depth literature overview.
The core principle of the theory is that well-invertible column subsets
exist under very general spectral assumptions on the input matrix.

In this paper, we focus on well-invertible column subsets of {\it proportional} size, namely the setting $|\sigma|=\Omega(k)$.
As a specific illustration, let $M$ be a $k\times n$ matrix ($n\geq k$) with unit-length columns.
The theory of restricted invertibility then asserts that,
under an extra assumption on the matrix norm, there exists a set of columns of $M$ of size $\Omega(k)$ such that the smallest singular value of the corresponding submatrix is bounded below by a positive constant.
Conversely, our goal in this work is to provide an {\it upper} bound on the number of well-invertible column subsets of $M$,
under a sparsity assumption on the input matrix.
In this sense, our result can be viewed as a dual to the Bourgain--Tzafriri theory: while restricted invertibility guarantees that good column selections exist, we show that, for sparse matrices, such selections are necessarily rare.
Our main result is the following asymptotic statement
about the relative proportion of well-invertible column subsets in sparse
matrices.
In what follows, by $M_{\cdot j}$ we denote $j$-th column of a matrix $M$.
Further, write
$\operatorname{supp}M_{\cdot j}$ for the support of its $j$-th column.

\begin{theor}[Main result: well-invertible column subsets of sparse matrices are rare]
\label{thm:weighted-local-sparse}
Let $M^{(k)}$, $k\geq 1$, be a deterministic sequence of $k\times n_k$ matrices with
$n_k/k\to\infty$. We assume the following:
\begin{enumerate}
  \item (Column support) The average column support size is of constant order:
  \[
  \frac1{n_k}\sum_{j=1}^{n_k}
  |\operatorname{supp}M^{(k)}_{\cdot j}|
  = O(1).
  \]
  \item (Average magnitude) The average magnitude of {\bf nonzero} entries
  is of constant order:
  \[
  \frac{
  1
  }{
     |\operatorname{supp}M^{(k)}|
    } 
    \sum_{i=1 }^k \sum_{j =1}^{n_k} |M^{(k)}_{ij}|
  =O(1).
  \]
  \item (Column overlaps)
  The total number of pairs of columns of $M^{(k)}$ whose supports overlap
  on two or more indices, is of order
  $$
  o\Big(\frac{n_k^2}{k}\Big).
  $$  
\end{enumerate}
Then for every $\varepsilon>0$, if
$\Omega_k\subset[n_k]$ are chosen uniformly at random with
$|\Omega_k|=\lfloor \varepsilon  k\rfloor$, we have \[
\smin\bigl(M^{(k)}_{\Omega_k}\bigr)=o(1)
\]
with probability $1-o(1)$. Equivalently, a uniformly chosen proportional-size
column subset is well-invertible with asymptotically vanishing probability.
\end{theor}

The ``columns overlap'' assumption in the above theorem
should not be thought of as a pseudorandomness condition on the
input matrix. Indeed, it is not difficult to verify
that for classical random matrix models of constant sparsity (including
combinatorial random matrices and SparseStack$^\top$),
the typical number of pairs of columns intersecting on two or more rows,
is of order $O(n_k^2/k^2)$, i.e almost $k$ times smaller than
the threshold we impose in the theorem.

We emphasize that our result is {\it universal}
rather than existential.
Of course, one can easily construct sparse matrices
satisfying the assumptions of the restricted invertibility theory,
and yet containing very few well invertible column subsets.
As an example, one can take a $k \times k$ identity matrix extended by the $k \times (n_k-k)$ zero matrix, or
concatenate $n_k/k$ identity matrices of size $k \times k$ to form a $k \times n_k$ matrix. 
In the first example, a random column subset of size $\Theta(k)$
is singular with high probability due to presence of zero columns.
In the second example, the obstruction to invertibility
is existence of two identical columns within the selected subset.

There are, however, sparse matrices for which counting well invertible column subsets is far less obvious. Such are, for example, random matrices with independent columns and $d$ uniformly chosen ones in each column, or random SparseStack$^\top$ matrices. As we will show, these matrices satisfy the assumptions of Theorem   \ref{thm:weighted-local-sparse} with high probability.

The proof of our theorem shows that, apart from the selected submatrix having a zero column or having two identical columns,
there is the third less obvious reason for the scarcity of well-invertible submatrices in the sparse setting. It is the abundance of $k \times m$ rectangular submatrices having very small least singular values,
and comprising certain combinatorial {\it tree} structure.
Here, $m =o(k)$ and $m \to \infty$ as $n_k/k \to \infty$. These matrices are defined in Section \ref{sec: tech overview} below. If a
selected
submatrix $M^{(k)}_{\Omega_k}$ of $M^{(k)}$ contains one of these special submatrices, then the least singular value of $M^{(k)}_{\Omega_k}$ will be small as well, and their profusion guarantees that the number of 
submatrices not containing any special $k \times m$ one is small.

Theorem  \ref{thm:weighted-local-sparse} directly addresses Problem~\ref{main problem}
for sparse matrices satisfying the local column-overlap and growth assumptions.
As a specific illustration of the theorem,
consider the SparseStack model of sketch matrices.
\begin{definition}[SparseStack]
\label{def:sparsestack}
Fix integers $n,\zeta,b \in \N$, and set
\[
    k:=\zeta b.
\]
Partition the column set $[k]$ into $\zeta$ consecutive blocks of size $b$. 
Independently choose a uniformly distributed random position in each block and populate these positions with independent $\pm 1$ entries.

More formally, the SparseStack matrix is a random $n\times k$ matrix
$S=(s_{ij})$
constructed as follows: for every row index $i \in [n]$ and every block index $t \in [\zeta]$, choose
$p_{i,t} \in [b],$
$\varrho_{i,t} \in \{-1,1\},$
independently and uniformly at random, and place the value $\varrho_{i,t}/\sqrt{\zeta}$ in the matrix coordinate of $S$
\[
    (i,(t-1)b+p_{i,t}).
\]
After the construction is complete, set all unassigned entries of $S$ to zero.
\end{definition}

The authors of \cite{CamanoEtAl2025} formulated the following conjecture:
\begin{conjecture}[{SparseStack: constant row sparsity;
see \cite[Conjecture~1.9]{CamanoEtAl2025}}]\label{main:conj}
A SparseStack test matrix is an $(r,1/2)$-OSI for some embedding dimension $k=O(r)$ and some sparsity level $\zeta=O(1)$.
\end{conjecture}

As an immediate corollary of the main theorem, we obtain
\begin{cor}[No OSI for constant-sparsity SparseStack model]
\label{cor:sparsestack-no-osi}
For every integer $\zeta\ge2$ and every choice of
$\varepsilon,\alpha>0$ there are constants $B_0,C>0$
depending only on $\zeta, \varepsilon,\alpha$
with the following property.
Let $b\geq B_0$, let $k=\zeta b$ and $n\geq C\,k$,
and assume that the $n\times k$ matrix $S$
is SparseStack with row
sparsity $\zeta$ and with block size $b$.
Let $\Omega$ be any non-random subset of $[n]$ of size at least $\varepsilon\,k$.
Then with probability at least $0.99$,
\[
\smin\left((S^\top)_{\Omega}\right)< \alpha,
\]
where $\smin\left((S^\top)_{\Omega}\right)$ is the 
smallest singular value of the $k\times |\Omega|$ submatrix $(S^\top)_\Omega$.
In particular, $S$ is not $(\varepsilon\,k,\alpha)$--OSI.
\end{cor}
\begin{remark}
The sign distribution of the nonzero entries in the SparseStack model can be replaced by an arbitrary distribution with finite first moment; our theorem then yields the same conclusion for $(S^\top)_\Omega$. 
We leave the details to the interested reader,
and refer to Sections~\ref{sec:matrix-theorem} and~\ref{sec:sparsestack}
for a discussion of quantitative aspects of our results.
\end{remark}

\begin{remark}
Our main result extends well beyond the SparseStack model. In particular, our theorem covers ``combinatorial'' random sketch matrices with i.i.d. rows, where the row supports are chosen uniformly at random among all $d$-element subsets of $[k]$ and the nonzero entries are $\pm 1$, as well as  random matrices with prescribed number of non-zero entries in every row and column.
\end{remark}

\subsection{Technical overview} \label{sec: tech overview} \,\\

We now explain the graph mechanism behind the main result. 
To illustrate the idea, in this discussion we consider matrices whose entries only take values in $\{0,1,-1\}$ and every column has exactly $d$ nonzero entries for some fixed $d\ge 2$.

\subsubsection{The obstruction pattern}
Given a rectangular matrix $M$, we view it as the adjacency matrix of a bipartite graph: rows are left vertices,
columns are right vertices, and each nonzero entry corresponds to an edge. 
Sampling column subsets
of $M$ is equivalent to constructing
a ``right-induced subgraph'', obtained as one-neighborhood of the selected right vertices. Notation-wise, if
$G=(I,J,E)$ is a bipartite graph on $I\sqcup J$ and $\Omega\subset J$, we write
\[
G[\Omega]:=(I,\Omega,E\cap(I\times\Omega)).
\]
The {\it local obstruction} to strong quantitative invertibility
of the random column subset, which we exploit in the proof, is an \mrt{m}.
Let $G[\Omega]$ be the support graph of
the sampled matrix $M_{\Omega}$. Suppose further that
$G[\Omega]$ contains an
induced depth-three rooted bipartite tree $T$ with right root $j_\star$. Thus
$j_\star$ has left neighbors $i_1,\dots,i_d$; for each $i_\ell$, there are
$m$ secondary (auxiliary) vertices ${\mathcal K}_{i_\ell}=\{c_{\ell,1},\dots,c_{\ell,m}\}$ adjacent to
$i_\ell$; and each auxiliary right vertex has its remaining $d-1$ left
neighbors in a third layer, with all these outer left neighbors distinct. 
Let us emphasize that the tree contains all graph neighbors of each of its right vertices. 

\begin{figure}[H]
\centering
\begin{tikzpicture}[
  scale=0.84,
  transform shape,
  x=1cm,
  y=1cm,
  rightv/.style={circle, draw=blue!65, fill=blue!10, thick, minimum size=6.5mm, inner sep=0pt},
  auxv/.style={circle, draw=blue!55, fill=blue!8, thick, minimum size=4.8mm, inner sep=0pt},
  leftv/.style={circle, draw=teal!70!black, fill=teal!10, thick, minimum size=6mm, inner sep=0pt},
  tailv/.style={circle, draw=gray!70, fill=gray!12, thick, minimum size=3.4mm, inner sep=0pt},
  edge/.style={draw=gray!70, line width=0.55pt},
  label/.style={font=\scriptsize}
]

\node[rightv] (root) at (0,0) {$j_\star$};

\node[leftv] (i1) at (-4.8,-1.15) {$i_1$};
\node[leftv] (i2) at (0,-1.15) {$i_2$};
\node[leftv] (i3) at (4.8,-1.15) {$i_3$};

\foreach \i in {1,2,3} {
  \draw[edge] (root) -- (i\i);
}

\node[auxv] (h11) at (-6.0,-2.35) {};
\node[auxv] (h12) at (-5.4,-2.35) {};
\node[auxv] (h13) at (-4.8,-2.35) {};
\node[auxv] (h14) at (-4.2,-2.35) {};
\node[auxv] (h15) at (-3.6,-2.35) {};

\node[auxv] (h21) at (-1.2,-2.35) {};
\node[auxv] (h22) at (-0.6,-2.35) {};
\node[auxv] (h23) at (0,-2.35) {};
\node[auxv] (h24) at (0.6,-2.35) {};
\node[auxv] (h25) at (1.2,-2.35) {};

\node[auxv] (h31) at (3.6,-2.35) {};
\node[auxv] (h32) at (4.2,-2.35) {};
\node[auxv] (h33) at (4.8,-2.35) {};
\node[auxv] (h34) at (5.4,-2.35) {};
\node[auxv] (h35) at (6.0,-2.35) {};

\foreach \aux in {h11,h12,h13,h14,h15} {\draw[edge] (i1) -- (\aux);}
\foreach \aux in {h21,h22,h23,h24,h25} {\draw[edge] (i2) -- (\aux);}
\foreach \aux in {h31,h32,h33,h34,h35} {\draw[edge] (i3) -- (\aux);}

\foreach \aux/\x in {
  h11/-6.0,h12/-5.4,h13/-4.8,h14/-4.2,h15/-3.6,
  h21/-1.2,h22/-0.6,h23/0,h24/0.6,h25/1.2,
  h31/3.6,h32/4.2,h33/4.8,h34/5.4,h35/6.0
} {
  \node[tailv] (\aux-L) at ({\x-0.18},-3.35) {};
  \node[tailv] (\aux-R) at ({\x+0.18},-3.35) {};
  \draw[edge] (\aux) -- (\aux-L);
  \draw[edge] (\aux) -- (\aux-R);
}

\node[label, above] at (0,0.35) {root right vertex};
\node[label, left] at (-6.8,-1.15) {$d=3$ left neighbors};
\node[label, left] at (-6.8,-2.35) {$m=5$ chosen right vertices per branch};
\node[label, left] at (-6.8,-3.35) {disjoint outer left neighborhoods};

\end{tikzpicture}

\vspace{0.6em}
\caption{A $m$-matching rooted tree, shown for right degree $d=3$ and width $m=5$. The four tree layers have $1$, $3$, $15$, and $30$ vertices, respectively.}
\label{fig:rooted-matching-tree}
\end{figure}
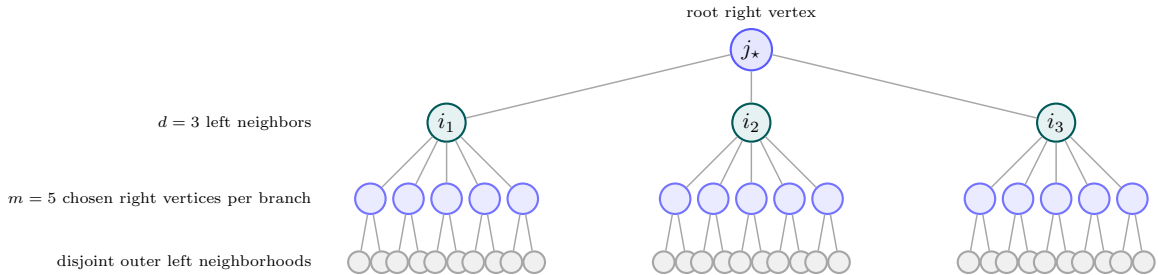

This pattern gives an explicit test vector $v$ supported on the right vertices
of $T$. Put $v_{j_\star}=1$, and choose the coefficients $v_{c_{\ell,a}}$
with $|v_{c_{\ell,a}}|=1/m$ so that the contribution of the root column
cancels on the rows $i_1,\dots,i_d$. The only remaining nonzero coordinates of
$M_{\Omega}v$ lie on the outer left layer. They are disjointly
supported, there are $dm(d-1)$ of them, and each has magnitude $1/m$. Hence
\[
\|M_{\Omega}v\|_2
\le
\frac{\sqrt{dm(d-1)}}{m}
\le
\frac d{\sqrt m}.
\]
Since $\|v\|_2\ge1$, this forces
$\smin(M_{\Omega})\le d/\sqrt m$.
Thus, whenever we are able to find the described pattern in $M_\Omega$
with parameter $m$ large, we immediately certify
a strong upper bound on the smallest singular value of $M_{\Omega}$.

\begin{figure}[H]
\centering
\resizebox{0.78\textwidth}{!}{%
\begin{tikzpicture}[
  x=0.336cm,
  y=0.288cm,
  font=\scriptsize,
  grid/.style={draw=gray!45,line width=0.25pt},
  rootcell/.style={fill=blue!55,opacity=0.28},
  levelonecell/.style={fill=teal!10},
  outercell/.style={fill=blue!55,opacity=0.28}
]
\node[anchor=east,font=\scriptsize] at (-2.7,-5) {$M_T=$};

\foreach \r/\lab in {
  0/{$i_1$},
  1/{$i_2$},
  2/{$i_3$},
  3/{$o_{1,1,1}$},
  4/{$o_{1,1,2}$},
  5/{$o_{1,2,1}$},
  6/{$o_{1,2,2}$},
  8/{$o_{3,5,1}$},
  9/{$o_{3,5,2}$}
} {
  \node[anchor=east] at (-0.25,-\r-0.5) {\lab};
}

\node at (0.5,0.55) {$j_\star$};
\node at (3.5,0.55) {$c_{1,\cdot}$};
\node at (8.5,0.55) {$c_{2,\cdot}$};
\node at (13.5,0.55) {$c_{3,\cdot}$};

\foreach \r in {0,1,2} {
  \fill[rootcell] (0,-\r) rectangle +(1,-1);
}
\foreach \c in {1,...,5} {
  \fill[levelonecell] (\c,0) rectangle +(1,-1);
}
\foreach \c in {6,...,10} {
  \fill[levelonecell] (\c,-1) rectangle +(1,-1);
}
\foreach \c in {11,...,15} {
  \fill[levelonecell] (\c,-2) rectangle +(1,-1);
}
\foreach \r/\c in {3/1,4/1,5/2,6/2,8/15,9/15} {
  \fill[outercell] (\c,-\r) rectangle +(1,-1);
}

\foreach \r in {0,...,9} {
  \foreach \c in {0,...,15} {
    \draw[grid] (\c,-\r) rectangle +(1,-1);
  }
}
\foreach \x in {0,1,6,11,16} {
  \draw[black!65,line width=0.5pt] (\x,0) -- (\x,-10);
}
\draw[black!65,line width=0.5pt] (0,0) -- (16,0);
\draw[black!65,line width=0.5pt] (0,-3) -- (16,-3);
\draw[black!65,line width=0.5pt] (0,-10) -- (16,-10);

\foreach \r in {0,1,2} {
  \node at (0.5,-\r-0.5) {$1$};
}
\foreach \c in {1,...,5} {
  \node at (\c+0.5,-0.5) {$1$};
}
\foreach \c in {6,...,10} {
  \node at (\c+0.5,-1.5) {$1$};
}
\foreach \c in {11,...,15} {
  \node at (\c+0.5,-2.5) {$1$};
}
\foreach \r/\c in {3/1,4/1,5/2,6/2,8/15,9/15} {
  \node at (\c+0.5,-\r-0.5) {$1$};
}
\node at (8.5,-7.5) {$\vdots$};
\end{tikzpicture}%
}
\vspace{0.6em}
\caption{The support-pattern matrix associated with the tree in Figure~\ref{fig:rooted-matching-tree}. Blank entries are zero; the displayed $1$'s record support entries. The first three rows are the root-neighbor rows where the test vector cancels, while the lower rows are disjoint outer-neighborhood rows.}
\label{fig:rooted-matching-tree-matrix}
\end{figure}

For a formal definition of the $m$--matching tree, see
Definition~\ref{def:rooted-matching-tree} below.
A generalized version of the deterministic singular-value statement sketched above is Proposition~\ref{prop:weighted-tree-small-smin}. 
The simplified description above only operates with the matrix support.
For a general {\it weighted}
matrix \(M\), the construction of the target tree pattern is more complicated: we require each edge connecting 
second and third layers of the tree,
to have a magnitude at least as large as 
that of the corresponding ``parent'' edge connecting the second layer to the root; this is the
weight-compatible condition in Definition~\ref{def:weight-compatible-rooted-matching-tree}.

\subsubsection{Reduction to $C_4$-free support graphs}
To extract an $m$--matching tree from a sampled subgraph, we repeatedly use a
two-stage sampling strategy.  Rather than looking only at the final sample, we
first expose a larger intermediate right vertex set.  More precisely, we may realize the
final uniform sample in two stages: first choose an intermediate set
\(\Omega'_k\subset[n_k]\) uniformly, with \(|\Omega'_k|\) larger than the final
sample size, and then choose the final set \(\Omega_k\subset\Omega'_k\)
uniformly with \(|\Omega_k|=r_k=\Theta(k)\). This intermediate set has enough room to reveal the desired structure, and after trimming a small number of ``bad'' right vertices it may have better deterministic properties than the original graph. In fact, we use multiple rounds of intermediate sampling in the proof. 

One of the key properties we obtained from an intermediate support graph is {\it \(C_4\)-freeness}. Here we illustrate that in detail. 
For distinct columns \(j,j'\), the overlap condition appearing in
Theorem~\ref{thm:weighted-local-sparse} has the following graph-theoretic
interpretation, where \(G^{(k)}\) is the support graph of \(M^{(k)}\):
\[
\bigl|\operatorname{supp}M^{(k)}_{\cdot j}
\cap
\operatorname{supp}M^{(k)}_{\cdot j'}\bigr|\ge2
\quad \Leftrightarrow \quad
\begin{gathered}
\text{there is a \(4\)-cycle \(C_4\) in the support graph \(G^{(k)}\)}\\
\text{containing \(j\) and \(j'\).}
\end{gathered}
\]
Specifically, if \(i,i'\) are two distinct common left neighbors of \(j\) and
\(j'\), then \(i-j-i'-j'-i\) is a four-cycle in the support graph. Thus the
third condition in the main theorem controls the number of pairs of right
vertices that lie together on a four-cycle.

Our construction relies on a somewhat technical \lsfci{} assumption on the graph, see Definition \ref{def:local-sparse-four-cycle-incidence}.
A cleaner, albeit slightly more restrictive, way to to ensure that it is satisfied is to assume that the
support graph has \(o(n^2/k)\) four-cycles in total:
\begin{equation}  \label{eq: few cycles}
\mbox{the number of }4\text{-cycles in }G^{(k)}
=
o\left(\frac{n_k^2}{k}\right).
\end{equation}
Throughout the introduction, the reader can substitute   condition \eqref{eq: few cycles} for \lsfci{}.
This easier assumption holds for a number of random matrix families,
and is used in the (non-technical) statement of the main result
above.
For example, in fixed-sparsity models such as SparseStack, the expected number 
of four-cycles is \(O_d(n_k^2/k^2)\), so this global condition holds with high
probability by Markov's inequality.

Under \lsfci{}, together with \(n_k/k\to\infty\) and average right degree
\(O(1)\), the intermediate sampled set can be trimmed to a large \(C_4\)-free
bounded-degree subgraph with high probability. From that point, the task is to
extract an \mrt{m} from the \(C_4\)-free subgraph.

The proof runs through the chain
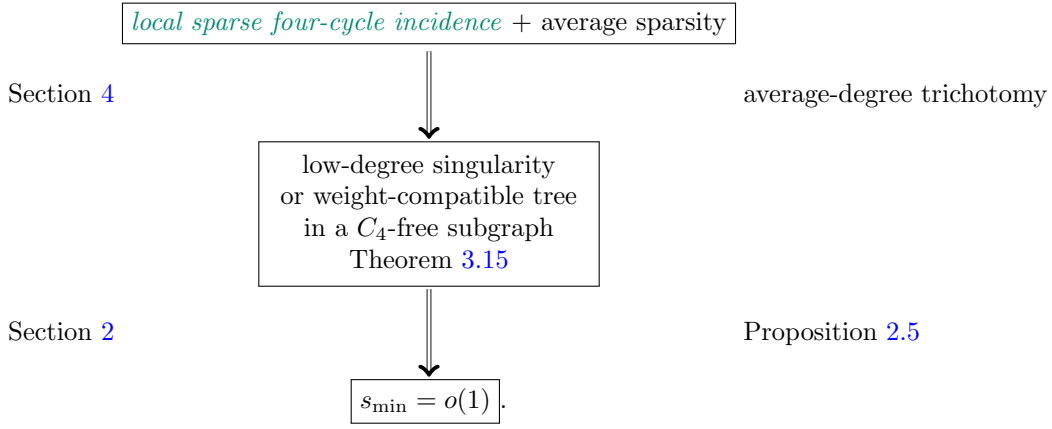
\vspace{0.3em}
\[
\begin{array}{r@{\;}c@{\;}l}
 & \boxed{\text{\lsfci{} + average sparsity}} & \\[0.35em]
\text{Section~\ref{sec:bridge-sampling}} & \flowdown &
\text{average-degree trichotomy}\\[0.35em]
 & \boxed{\begin{array}{c}
\text{low-degree singularity}\\
\text{or weight-compatible tree}\\
\text{in a }C_4\text{-free subgraph} \\
\text{Theorem~\ref{thm:c4-finite-subgraph}}
\end{array}} & \\[1.3em]
\text{Section~\ref{sec:tree-test-vectors}} & \flowdown &
\text{Proposition~\ref{prop:weighted-tree-small-smin}}\\[0.35em]
 & \boxed{\smin=o(1)}. &
\end{array}
\]
\vspace{0.3em}

The key technical result of our work is the sampling theorem for \(C_4\)-free
graphs in Section~\ref{sec:c4-sampling-subgraph}; here we state an informal
version:
\begin{theor*}[Informal statement of Theorem~\ref{thm:c4-finite-subgraph}]
Fix \(d\ge2\) and $\rho >0$. Consider three positive integers $r,k,n$ satisfying 
\( \rho k \le r\le k\le n.\) Let \(G=(I,J,M)\) be a \(C_4\)-free weighted bipartite graph with
\(
|I|=k,\qquad |J|=n,
\)
and right support degrees between \(2\) and \(d\).
Set
\[
\Lambda:=\min\left( \frac r2,\frac nk\right)
\quad\text{and}\quad
 m = \Theta\left( \frac{1}{d^3}\frac{\log\Lambda}{\log(\frac{d}{\rho}\log\Lambda)} \right).
\]
Then for a uniform right sample $\Omega$ of size $r$, we have
\[
\mathbb P\left(
\text{$G[\Omega]$ contains a \wcmrt{m}}
\right) \ge 1- O( \tfrac{d}{\log(\Lambda)}).
\]
\end{theor*}
The notion of weight-compatible above was a technical requirement to ensure that the test vector construction works in the weighted case.  
Further, we have the following reduction from \lsfci{} to the $C_4$-free support subgraph.

\begin{prop*}[Average-degree trichotomy, asymptotic form of Proposition~\ref{prop:average-degree-bridge-alternatives}]
\label{prop:intro-bridge}
Let $G_k$ be a sequence of weighted bipartite graphs with $k$ left vertices,
$n_k$ right vertices, average right support-degree $O(1)$, average nonzero
weight magnitude at most $1$, $n_k/k\to\infty$, and \lsfci{}. Let $r_k$
satisfy $k/r_k=O(1)$. Then, for a uniform right sample $\Omega_k$ of size $r_k$, with
probability $1-o(1)$ one of the following holds:
\begin{enumerate}
\item $G_k[\Omega_k]$ has an isolated sampled right vertex;
\item $G_k[\Omega_k]$ has two degree-one sampled right vertices with the same
left neighbor;
\item $G_k[\Omega_k]$ contains a \wcmrt{m_k} inside a
$C_4$-free subgraph of bounded maximum right degree, with nonzero weights
bounded by a constant depending only on the average-degree bound, where
\[
m_k\to\infty.
\]
\end{enumerate}
\end{prop*}

The first two cases are the price of working with average right degree, rather than a model in which every sampled column has degree $d$. Average control can leave many degree-zero or degree-one columns.
If the sample contains an isolated right vertex, or two degree-one right
vertices supported on the same left vertex, the corresponding sampled matrix is
already singular. If neither low-degree obstruction occurs, the trichotomy argument
truncates to a bounded-degree subgraph and produces the tree case.
The bounded-degree cutoff gives at most $O(m_k)$ third-layer vertices, so
$m_k\to\infty$ forces the desired small singular value.

\subsubsection{Extracting the $m$-matching rooted tree from a $C_4$-free subgraph}

\paragraph{\bf{Witness blocks.}}
The $C_4$-free property helps us find such trees through a simpler combinatorial object.
The rough idea is as follows. We still have the root vertex $j_\star$ and its
left neighbors $I_\star=N_G(j_\star)$. To build an \mrt{m}
directly, one would need to find disjoint right-neighbor subsets
${\mathcal K}_i \subseteq N_G(i)\setminus\{j_\star\}$ for each $i \in I_\star$
so that all the outer left neighborhoods $N_G({\mathcal K}_i)\setminus\{i\}$
are disjoint. Instead, we first search for larger auxiliary right-neighbor
sets ${\mathcal R}_i$ of each $i \in I_\star$, each of size at least
$m(1+(d-1)^2)$, which we call \reservoir{}s.  The $C_4$-freeness forces the
\reservoir{}s to be disjoint regardless of how they are chosen, and
$C_4$-freeness together with the right-degree bound ensures that for every $j \in \cup_{i \in I_\star} {\mathcal R}_i$, the number of other $j'\in \cup_{i \in I_\star} {\mathcal R}_i$, sharing a neighbor with $j$ outside of $I_\star$, is bounded. Thus we can greedily prune the \reservoir{}s to find
${\mathcal K}_i \subseteq {\mathcal R}_i$ for each $i \in I_\star$ so that all
the outer left neighborhoods $N_G({\mathcal K}_i)\setminus\{i\}$ are disjoint,
forming the desired \mrt{m}. See
Lemma~\ref{lem:reservoirs-produce-tree} for the formal statement.

The combination of the root and its \reservoir{}s is the intermediate object
we want the random sample to capture. More precisely, if each
${\mathcal R}_i$, $i\in I_\star$, has size at least
$m(1+(d-1)^2)$, then the right-vertex set
\[
{\mathcal W}(j_\star):=\{j_\star\}\cup
\bigcup_{i \in I_\star}{\mathcal R}_i
\]
is a \witnessblock{}. Capturing this whole block in a uniform sample is easier
to analyze than capturing the already-pruned tree.

\paragraph{\bf Overall strategy in the $C_4$-free case.}
Now suppose that $G=(I,J,E)$ is $C_4$-free, with $|I|=k$ and $|J|=n$, where
$n$ is much larger than $k$. We choose a uniform random subset
$\Omega \subseteq J$ of size $r$, with $r$ of order $k$, and want to show that
with high probability the induced subgraph $G[\Omega]$ contains a
\witnessblock{}.

To justify the existence of many
pairwise disjoint \witnessblock{}s, we rely on the property that typical pair of right vertices don't share common left neighors to guarantee the existence of many root candidates with good properties. Then a multi-matching argument chooses pairwise disjoint heavy \reservoir{}s around
their neighboring left vertices, producing many disjoint \witnessblock{}s.

Suppose $\lambda=n/r$ and we have $K$ disjoint
\witnessblock{}s, each of size at most $B$. A fixed block is then contained in
a uniform $r$-subset with probability roughly $\lambda^{-B}$, and formally at
least $(8\lambda)^{-B}$ in Lemma~\ref{lem:capture-disjoint-blocks}. Because the
blocks are disjoint, sequential exposure gives the heuristic failure bound
\[
(1-c\lambda^{-B})^K,
\]
With compatible choices of $\lambda$, $B$, and $K$, this probability is very small. This is the main idea behind the proof of
Theorem~\ref{thm:c4-finite-subgraph}.

\subsection{Organization}

In Section~\ref{sec:tree-test-vectors} we prove that matching rooted trees force small
least singular values.  In Section~\ref{sec:c4-sampling-subgraph} we prove the
$C_4$-free sampling theorem.  In Section~\ref{sec:bridge-sampling} we prove the
average-degree trichotomy, reducing \lsfci{} to the $C_4$-free subgraph case and the
low-degree singular cases.  In Section~\ref{sec:matrix-theorem} we combine
the graph cases with the test vectors to prove the matrix theorem.
Section~\ref{sec:sparsestack} applies the result to SparseStack and OSI. We
close with open problems.

\bigskip

{\bf Funding acknowledgement.}
K.T. was partially supported by NSF grant DMS 2452120.

{\bf Acknowledgement of AI Assistance.} The authors used ChatGPT for language editing, stylistic suggestions, and assistance in developing and checking some proof arguments during the preparation of this manuscript. All mathematical statements, proofs, and final wording were independently reviewed and verified by the authors, who take full responsibility for the content of the paper.

\section{From matrices to graphs}
\label{sec:tree-test-vectors}

This section records the analytic obstruction used throughout the paper.  The
object is purely local in the support of a weighted bipartite graph: once a
sampled support contains a weight-compatible matching rooted tree, the
corresponding weighted matrix has an explicit test vector with small image.

\subsection{Graph Conventions and Matching Rooted Trees}
A matrix can also be viewed as a weighted bipartite graph. We will use the following conventions throughout the paper.
\begin{definition}[Weighted bipartite graph and support graph]
\label{def:weighted-bipartite-graph}
Let $G=(I,J,M)$ be a weighted bipartite graph, where
\[
M\in\mathbb R^{I\times J}.
\]
We call $I$ the left vertex set and $J$ the right vertex set. We declare that
\begin{align}
  \label{eq:edge-weight}
(i,j) \in I\times J \mbox{ is an edge}
\quad
\Leftrightarrow
\quad
M_{ij}\ne0
\end{align}
and, when $(i,j)$ is an edge, its weight is $M_{ij}$.
With this convention, we define the support of $G$ to be the set of edges
\[
E_M:=\{(i,j)\in I\times J:M_{ij}\ne0\}.
\]
By the support graph of $G$ we mean the unweighted bipartite graph $(I,J,E_M)$.
\end{definition}
Indeed, many technical statements in the paper only depend on the support, and not on the actual weights. We will often blur the distinction between $G$ and its support graph $(I,J,E_M)$. In other words, the majority of the arguments are purely combinatorial, and the weights only come into play when we construct the test vector for the matching rooted tree.

For a vertex $v$, write $N_G(v)$ for its neighbors in this support graph. Degrees always mean support degrees.
For $\Omega\subseteq J$, write
\[
G[\Omega]:=(I,\Omega,M_{\cdot,\Omega})
\]
for the right-induced weighted subgraph on $\Omega$.

The following dictionary will be used repeatedly.
\[
\begin{array}{@{}ll@{}}
\textbf{Matrix language} & \textbf{Graph language} \\
\text{row of }M & \text{left vertex} \\
\text{column of }M & \text{right vertex} \\
\text{sampled columns }M_\Omega & \text{right-induced weighted subgraph }G[\Omega] \\
\text{column sparsity} & \text{right degree} \\
\text{zero sampled column} & \text{isolated sampled right vertex} \\
\text{two degree-one columns on one row} &
  \text{two degree-one right vertices sharing the same left neighbor} \\
\end{array}
\]

\begin{definition}[$m$-Matching Rooted Tree]
\label{def:rooted-matching-tree}

Let $G=(I,J,M)$ be a weighted bipartite graph. Fix a right vertex $j_\star\in J$, 
and set
\[
I_\star:=N_G(j_\star).
\]
We say that $G$ contains an $m$-matching rooted tree at $j_\star$ if, for each $i\in I_\star$, there is a set
\[
\mathcal K_i\subseteq N_G(i)\setminus\{j_\star\},
\qquad
|\mathcal K_i|=m,
\]
such that
\[
N_G(j)\cap I_\star=\{i\}
\qquad
\forall i\in I_\star,\ \forall j\in\mathcal K_i,
\]
the sets $\mathcal K_i$, $i\in I_\star$, are pairwise disjoint, and the outer left neighborhoods
\[
N_G(j)\setminus I_\star,
\qquad
j\in\bigcup_{i\in I_\star}\mathcal K_i,
\]
are pairwise disjoint. Indeed, these vertices form a three-layer rooted tree in the support graph, including all neighbors of the right vertices in the tree:
\begin{itemize}
  \item $j_\star$ is the root,
  \item $I_\star$ is the first left layer, 
  \item the vertices in $\bigcup_i\mathcal K_i$ are the second right layer, and 
  \item the sets $N_G(j)\setminus I_\star$ for $j \in \bigcup_{i\in I_\star}\mathcal K_i$ form the third left layer.
\end{itemize}

\end{definition}

\begin{remark}[Automatic disjointness in the $C_4$-free case]
If $G$ is $C_4$-free, then the condition $N_G(j)\cap I_\star=\{i\}$ and the
pairwise disjointness of the sets $\mathcal K_i$, $i\in I_\star$, are
automatic. Indeed, any violation would give a second-layer right vertex $j$
and the root $j_\star$ two common left neighbors.
\end{remark}

\begin{definition}[Weight-compatible matching rooted tree]
\label{def:weight-compatible-rooted-matching-tree}
Let $G=(I,J,M)$ contain an \mrt{m} at $j_\star$, witnessed by
sets $\mathcal K_i$, $i\in I_\star=N_G(j_\star)$. We say that this rooted tree
is \emph{weight-compatible} if
\[
|M_{i,j}|\ge |M_{i,j_\star}|
\qquad
\forall i\in I_\star,\ \forall j\in\mathcal K_i.
\]
We say that $G$ contains a weight-compatible $m$-matching rooted tree if some
choice of root and witness sets has this property.
\end{definition}

\subsection{Weighted Graphs}

\begin{prop}[Weight-Compatible Matching Rooted Tree Implies Small Singular Value]
\label{prop:weighted-tree-small-smin}

Let $G=(I,J,M)$ be a real weighted bipartite graph. Suppose $G$ contains a
\wcmrt{m}, witnessed by a root $j_\star$ and sets
$\mathcal K_i$, $i\in I_\star=N_G(j_\star)$. 
Define the third left layer
\[
\mathcal L_{\rm out}:=
\bigcup_{i\in I_\star}
\bigcup_{j\in\mathcal K_i}
\left(N_G(j)\setminus I_\star\right),
\]
and define its maximum weight by
\[
A_{\rm out}:=
\max\left\{
|M_{\ell,j}|:
i\in I_\star,\ j\in\mathcal K_i,\ \ell\in N_G(j)\setminus I_\star
\right\},
\]
with the convention $A_{\rm out}=0$ if the set is empty. Then
\[
\smin(M)\le
\frac{A_{\rm out}}{m}\sqrt{|\mathcal L_{\rm out}|}.
\]

\end{prop}

\begin{proof}

For $j\in\bigcup_{i\in I_\star}\mathcal K_i$, let $i(j)$ be the unique vertex
of $I_\star$ with $j\in\mathcal K_{i(j)}$. By the definition of an
\mrt{m}, $j$ is adjacent to no other vertex of $I_\star$.

Define $v\in\mathbb R^J$ by
\[
v_{j_\star}=1,\qquad
v_j=-\frac{M_{i(j),j_\star}}{mM_{i(j),j}}
\quad\text{for }j\in\bigcup_{i\in I_\star}\mathcal K_i,
\]
and put $v_j=0$ on all remaining right vertices. The denominators are nonzero
because $j\in N_G(i(j))$. Also $\|v\|_2\ge1$, and weight compatibility gives
$|v_j|\le1/m$ for every selected second-layer right vertex.

For every $i\in I_\star$, the second-layer contributions cancel the root
contribution:
\[
(Mv)_i
=
M_{i,j_\star}
\;+\;
\sum_{j\in\mathcal K_i}M_{i,j}
\left(-\frac{M_{i,j_\star}}{mM_{i,j}}\right)
=0.
\]

For $\ell\notin I_\star$, the root contributes nothing. The outer left
neighborhoods of the selected right vertices are pairwise disjoint, so at most
one selected right vertex contributes to row $\ell$. Hence every nonzero
coordinate of $Mv$ outside $I_\star$ has magnitude at most $A_{\rm out}/m$.

Therefore
\[
\|Mv\|_2^2
\le
|\mathcal L_{\rm out}|\frac{A_{\rm out}^2}{m^2}.
\]
Since $\|v\|_2\ge1$, this proves the claimed bound.

\end{proof}

\begin{remark}[Signed right-regular case]
If the support is right $d$-regular and all nonzero entries of $M$ are signs,
then $A_{\rm out}=1$ unless the third left layer is empty, in which case the
bound is trivial. Moreover, the root has $d$ first-layer neighbors, so there
are $md$ selected second-layer right vertices. Each has at most $d-1$ outer
left neighbors, and therefore
\[
|\mathcal L_{\rm out}|\le md(d-1).
\]
Thus Proposition~\ref{prop:weighted-tree-small-smin} gives
\[
\smin(M)\le \sqrt{\frac{d(d-1)}m}\le \frac d{\sqrt m}.
\]
\end{remark}

\section{The \texorpdfstring{$C_4$}{C4}-free Sampling Theorem}
\label{sec:c4-sampling-subgraph}

\subsection{Setup and Proof Roadmap}
In this section we use the graph conventions from Section~\ref{sec:tree-test-vectors}. The goal is to show that a uniform sample of $r$ right vertices typically contains a \wcmrt{m}, where $m$ grows with $r$ and $n/k$. (See the informal statement of Theorem~\ref{thm:c4-finite-subgraph} in the introduction, and the precise statement later in this section below.)

We say that a bipartite graph $G$ is $C_4$-free if it has no $4$-cycles, or equivalently if no two distinct right vertices have two common left neighbors:
\[
|N_G(j)\cap N_G(j')|\le1
\qquad
\forall j\ne j'\in J.
\]
Thus two right vertices can share at most one left neighbor.
\phantomsection\label{def:overlap-parameter}
When $|J|\ge2$, we measure how often such a common neighbor occurs by the
\overlapparameter{}
\[
\eta(G):=
\mathbb P\bigl(N_G(\mathsf j_1)\cap N_G(\mathsf j_2)\ne\emptyset\bigr),
\]
where $(\mathsf j_1,\mathsf j_2)$ is a uniformly chosen ordered pair of distinct right vertices.
For bounded right degrees and \(n/k\to\infty\), \(C_4\)-freeness forces
\(\eta(G)=o(1)\); see Lemma~\ref{lem:c4-sparse-overlap}.

\begin{definition}[Roots, reservoirs, and witness blocks]
\label{def:c4-witness-block}
In this section a root, or candidate root, is a right vertex $j_\star\in J$.
Its first left layer is
\[
I_\star:=N_G(j_\star).
\]
For $i\in I_\star$, an $L$-\reservoir{} at $i$ is a set
\[
\mathcal R_i\subseteq N_G(i)\setminus\{j_\star\}
\]
of auxiliary right vertices with $|\mathcal R_i|\ge L$. 
After choosing, for each $i\in I_\star$, an $L$-\reservoir{} $\mathcal R_i$,
the associated $L$-\witnessblock{} at $j_\star$ is the subset of right vertices
\[
{\mathcal W}(j_\star):=\{j_\star\}\cup\bigcup_{i\in I_\star}\mathcal R_i.
\]
The \reservoir{}s $\mathcal R_i$, $i\in I_\star$, witness this block.
\end{definition}
\begin{remark}
  Due to the $C_4$-free condition, the \reservoir{}s ${\mathcal R}_i$, $i \in I_\star$, are pairwise disjoint. 
\end{remark}

For a given positive integer $d \ge 2$, let
$$
  \beta_d := 1+(d-1)^2.
$$
If we have an $L$-\witnessblock{} at $j_\star$ with $L\ge \beta_d m$, then the pruning lemma, Lemma~\ref{lem:reservoirs-produce-tree}, relies on the $C_4$-free condition to greedily choose subfamilies
$\mathcal K_i\subseteq\mathcal R_i$ and produces the required \mrt{m} at $j_\star$. The weighted refinement below chooses \light{} roots and heavy \reservoir{}s so that this tree is automatically weight-compatible.
Informally, a heavy \reservoir{} at a left vertex consists of right neighbors with
large incident weights, while a \light{} root is not among those heavy neighbors
at any adjacent left vertex; the precise definition is
Definition~\ref{def:heavy-neighbor-light-root}.

A high level overview of the proof is as follows. We first show that a random $\lambda r$-subset of the right vertices contains many \light{} candidate roots with large heavy \reservoir{}s, where $\lambda>2$ is a parameter that grows slowly with $r$ and $n/k$. Then we show that a uniform $r$-subset of the right vertices captures at least one complete \witnessblock{} with high probability. Finally, the pruning lemma and the heavy-reservoir corollary produce a \wcmrt{m} from this captured block,
establishing the main result. 

The proof uses four scales. First, a small \overlapparameter{} produces many candidate right roots. The number of candidate roots is denoted by $K$; its allowable size is limited by the \overlapparameter{} $\eta(G)$ and by the final sample size $r$.

Second, for a target width $m$, we build heavy \reservoir{}s around the left neighbors of each candidate root. The \reservoir{} threshold is
\[
L=m\beta_d = m\bigl(1+(d-1)^2\bigr).
\]

Third, each root together with its \reservoir{}s forms a \witnessblock{}. The size of one block is at most
\[
B=1+dL.
\]
Thus $B$ is of order $d^3 m$.

Fourth, the random sample must capture one complete \witnessblock{}. If the ambient right side has size about $\lambda r$, then a fixed block of size $B$ is captured with probability at least $(8\lambda)^{-B}$. Intuitively, with $K$ disjoint attempts, the failure probability is therefore bounded by
\[
(1-(8\lambda)^{-B})^K\le \exp\{-K(8\lambda)^{-B}\}.
\]
The convenient finite condition
\[
(8\lambda)^{2B}\le K
\]
turns this into the cleaner error bound $\exp(-\sqrt K)$.

The exact parameter configurations in this section are chosen only to make these four steps compatible.

\begin{lemma}[Basic consequences of $C_4$-freeness]
\label{lem:c4-basic}

Let $G=(I,J,M)$ be $C_4$-free.

\begin{enumerate}
\item If $i\in I$ and $j,j'\in N_G(i)$ are distinct, then
\[
(N_G(j)\setminus\{i\})\cap (N_G(j')\setminus\{i\})=\emptyset.
\]

\item If $j_\star\in J$, $I_\star=N_G(j_\star)$, and
\[
j\ne j_\star,
\qquad
N_G(j)\cap I_\star\ne\emptyset,
\]
then $j$ is adjacent to a unique vertex of $I_\star$.

\item If $J'\subseteq J$, then $G[J']$ is $C_4$-free.

\end{enumerate}

\end{lemma}

\begin{proof}
For the first assertion, if a left vertex $x\ne i$ belonged to both $N_G(j)$ and $N_G(j')$, then $j$ and $j'$ would share both $i$ and $x$, contradicting $C_4$-freeness.

For the second assertion, if $j$ were adjacent to two distinct vertices of $I_\star$, then $j$ and $j_\star$ would have two common left neighbors. This again contradicts $C_4$-freeness.

The third assertion is immediate because deleting right vertices cannot create a new pair of right vertices with two common left neighbors.

\end{proof}

\begin{lemma}[$C_4$-free graphs have sparse overlap]
\label{lem:c4-sparse-overlap}

Let $G=(I,J,M)$ be $C_4$-free with $|I|=k$, $|J|=n\ge2$, and every right degree between $2$ and $d$. Then
\[
\eta(G)\le 2d\,\frac{k}{n}.
\]

Consequently, if $d$ is fixed and $G^{(k)}$ is a sequence of such graphs with $n_k/k\to\infty$, then $\eta(G^{(k)})=o(1)$.

\end{lemma}

\begin{proof}
For each left vertex $i\in I$, let $D_i:=|N_G(i)|$. An ordered pair $(j_1,j_2)$ with $j_1\ne j_2$ shares a left neighbor only if both vertices belong to $N_G(i)$ for some $i$. Hence the number of such ordered pairs is at most
\[
\sum_{i\in I}D_i(D_i-1).
\]

We claim $D_i\le k-1$ for every $i$. Indeed, by the first basic consequence of $C_4$-freeness, the sets $N_G(j)\setminus\{i\}$, $j\in N_G(i)$, are pairwise disjoint. Since every right vertex has degree at least $2$, each of these sets is nonempty, so $D_i\le k-1$.

Since every right vertex has degree at most $d$, $\sum_iD_i\le dn$. Therefore
\[
\sum_iD_i(D_i-1)\le \sum_iD_i^2\le (k-1)dn.
\]
Dividing by the total number $n(n-1)$ of distinct ordered pairs gives
\[
\eta(G)\le \frac{d(k-1)}{n-1}\le 2d\,\frac{k}{n}.
\]

\end{proof}

\subsection{Three Finite Ingredients}

The proof of the sampling theorem uses three finite ingredients: a reservoir-pruning lemma, a root-selection lemma, and a block-capture lemma.

The first ingredient is the local $C_4$-free pruning step. Once every
neighbor of the root has a large \reservoir{}, it extracts the disjoint outer
neighborhoods required in the matching rooted tree.  

\begin{lemma}[Reservoirs produce an $m$-matching rooted tree]
\label{lem:reservoirs-produce-tree}

Fix integers $d\ge1$ and $m\ge1$, and set
\[
\beta_d:=1+(d-1)^2.
\]
Let $G=(I,J,M)$ be $C_4$-free and suppose every right vertex has degree at most $d$. Fix $j_\star\in J$, set $I_\star=N_G(j_\star)$, and assume that for every $i\in I_\star$ there is a \reservoir{}
\[
\mathcal R_i\subseteq N_G(i)\setminus\{j_\star\}
\]
with
\[
|\mathcal R_i|\ge m\beta_d.
\]
Then $G$ contains an \mrt{m} at $j_\star$ with $\mathcal K_i\subseteq\mathcal R_i$ for every $i\in I_\star$.

\end{lemma}

\begin{figure}[H]
\centering
\begin{tikzpicture}[
  scale=0.82,
  transform shape,
  x=1cm,
  y=1cm,
  rightv/.style={circle, draw=blue!65, fill=blue!10, thick, minimum size=6.5mm, inner sep=0pt},
  chosen/.style={circle, draw=blue!80!black, fill=blue!22, very thick, minimum size=4.8mm, inner sep=0pt},
  unused/.style={circle, draw=gray!60, fill=gray!10, thick, minimum size=3.1mm, inner sep=0pt},
  leftv/.style={circle, draw=teal!70!black, fill=teal!10, thick, minimum size=6mm, inner sep=0pt},
  outer/.style={circle, draw=gray!70, fill=gray!12, thick, minimum size=3.2mm, inner sep=0pt},
  collision/.style={circle, draw=gray!70!black, fill=gray!12, thick, minimum size=3.2mm, inner sep=0pt},
  edge/.style={draw=gray!70, line width=0.55pt},
  goodedge/.style={draw=blue!55, line width=0.75pt},
  badedge/.style={draw=red!70!black, dashed, line width=0.75pt},
  fadededge/.style={draw=gray!45, line width=0.35pt},
  box/.style={draw=gray!55, dashed, rounded corners=2pt, fill=gray!4},
  label/.style={font=\scriptsize}
]

\node[rightv] (root) at (0,0) {$j_\star$};
\node[leftv] (i1) at (-5.4,-1.1) {$i_1$};
\node[leftv] (i2) at (0,-1.1) {$i_2$};
\node[leftv] (i3) at (5.4,-1.1) {$i_3$};

\foreach \i in {1,2,3} {\draw[edge] (root) -- (i\i);}

\foreach \base/\idx/\inode in {-5.4/1/i1,0/2/i2,5.4/3/i3} {
  \draw[box] ({\base-1.7},-4.05) rectangle ({\base+1.7},-1.75);
  \node[label, anchor=south west] at ({\base-1.65},-1.72) {$\mathcal R_{i_\idx}$};

  \foreach \row in {1,2,3,4} {
    \foreach \col in {1,2,3,4,5} {
      \node[unused] (u\idx-\row-\col) at ({\base-1.2+0.6*(\col-1)},{-2.0-0.34*(\row-1)}) {};
      \draw[fadededge] (\inode) -- (u\idx-\row-\col);
    }
  }

  \foreach \col in {1,2,3,4,5} {
    \node[chosen] (a\idx\col) at ({\base-1.2+0.6*(\col-1)},-3.55) {};
    \draw[goodedge] (\inode) -- (a\idx\col);
  }
}

\foreach \aux/\x in {
  a11/-6.6,a12/-6.0,a13/-5.4,a14/-4.8,a15/-4.2,
  a21/-1.2,a22/-0.6,a23/0,a24/0.6,a25/1.2,
  a31/4.2,a32/4.8,a33/5.4,a34/6.0,a35/6.6
} {
  \node[outer] (\aux-o1) at ({\x-0.17},-4.65) {};
  \node[outer] (\aux-o2) at ({\x+0.17},-4.65) {};
  \draw[goodedge] (\aux) -- (\aux-o1);
  \draw[goodedge] (\aux) -- (\aux-o2);
}

\node[collision] (c1) at (-4.03,-4.65) {};
\draw[badedge] (u2-3-1) -- (c1);
\node[collision] (c2) at (1.38,-4.65) {};
\draw[badedge] (u3-3-1) -- (c2);
\node[label, red!70!black] at (0,-5.08) {discarded candidates may collide with previously exposed outer neighborhoods};

\node[label, above] at (0,0.35) {root};
\node[label, left] at (-7.45,-2.9) {gray reservoir};
\node[label, left] at (-7.45,-3.55) {blue selected $\mathcal K_i$};
\node[label, below] at (0,-5.45) {selected vertices have disjoint outer left neighborhoods, giving the same $d=3,m=5$ tree shape as in Figure~\ref{fig:rooted-matching-tree}};

\end{tikzpicture}
\caption{Reservoir pruning around a root for $d=3$ and $m=5$. Each left neighbor of the root has a reservoir of size $m(1+(d-1)^2)=25$. The pruning lemma selects five blue vertices from each reservoir and avoids candidates whose outer left neighborhoods collide with previously selected vertices.}
\label{fig:reservoir-pruning}
\end{figure}
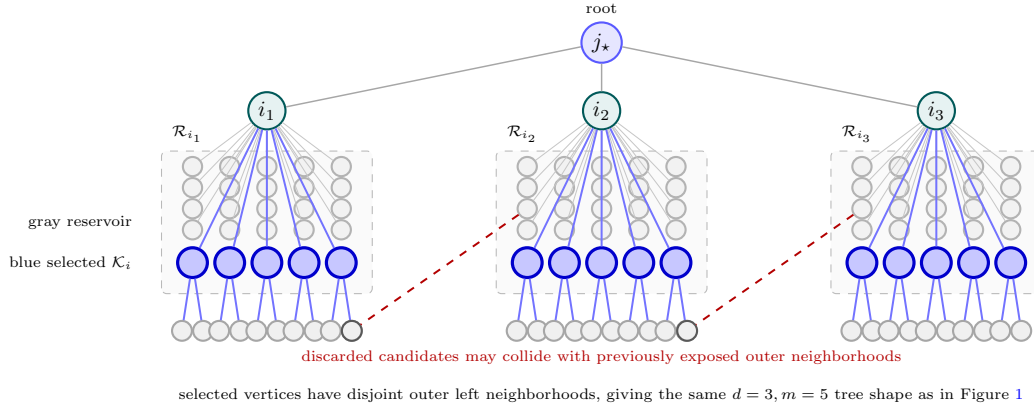

\begin{proof}
Because $G$ is $C_4$-free, every $j\in\mathcal R_i$ satisfies
\[
N_G(j)\cap I_\star=\{i\}.
\]
Otherwise $j$ and $j_\star$ would have two common left neighbors. In particular, the \reservoir{}s $\mathcal R_i$ are pairwise disjoint as $i$ varies.

For $j\in\bigcup_{i\in I_\star}\mathcal R_i$, call $N_G(j)\setminus I_\star$ its outer left neighborhood. Enumerate $I_\star=\{i_1,\dots,i_q\}$, where $q\le d$, and choose $\mathcal K_{i_s}\subseteq\mathcal R_{i_s}$ greedily.

Suppose $\mathcal K_{i_1},\dots,\mathcal K_{i_{s-1}}$ have already been chosen. Fix a previously chosen vertex $j'$. Its outer left neighborhood has size at most $d-1$. For each $x\in N_G(j')\setminus I_\star$, there is at most one vertex $j\in\mathcal R_{i_s}$ with $x\in N_G(j)$; two such vertices would share both $i_s$ and $x$, creating a $4$-cycle.

Thus each previously chosen vertex rules out at most $d-1$ candidates from $\mathcal R_{i_s}$. Since there are at most $(s-1)m\le(d-1)m$ previously chosen vertices, at most $m(d-1)^2$ candidates are ruled out. The \reservoir{} size assumption
\[
|\mathcal R_{i_s}|\ge m\beta_d=m+m(d-1)^2
\]
leaves at least $m$ available candidates; take these as $\mathcal K_{i_s}$.
The outer left neighborhoods of distinct vertices inside the same $\mathcal R_{i_s}$ are automatically disjoint: otherwise two such right vertices would share both $i_s$ and an outer left vertex, contradicting $C_4$-freeness.

After all \reservoir{}s are processed, the chosen subfamilies have size $m$ and have pairwise disjoint outer left neighborhoods. Hence they form an \mrt{m} at $j_\star$.
See Figure \ref{fig:reservoir-pruning} illustrating the reservoir pruning procedure.
\end{proof}

\begin{definition}[Heavy-neighbor sets and light roots]
\label{def:heavy-neighbor-light-root}
For a $G=(I,J,M)$ with maximum right degree $d$. 
For $i\in I$, write
\[
D_i:=|N_G(i)|.
\]
Let
\(
N^\dagger_G(i)\subseteq N_G(i)
\)
be any set of cardinality $\lceil \tfrac{1}{32d} D_i\rceil$, 
such that
\[
|M_{ij}|\ge |M_{ij'}|
\qquad
\forall j\in N^\dagger_G(i),\ 
\forall j'\in N_G(i)\setminus N^\dagger_G(i).
\]
A right vertex $j_\star\in J$ is \light{} if
\[
j_\star\notin N^\dagger_G(i)
\qquad
\forall i\in N_G(j_\star).
\]
\end{definition}
\begin{remark}[Graph combinatorics versus matrix weights]
The choice of the sets $N^\dagger_G(i)$ is the only place where the proof uses
information beyond the combinatorics of the support graph. Indeed, after these
sets are fixed, the rest of the argument is purely graph-theoretic. In
particular, for $\{0,\pm1\}$ matrices this weighted choice is unnecessary:
all nonzero entries have the same magnitude, so any support-theoretic choice of
the corresponding auxiliary neighbor sets already gives the required
weight-compatibility.
\end{remark}
\begin{cor}[Heavy reservoirs produce a weight-compatible tree]
\label{cor:heavy-reservoirs-weight-compatible}
Fix integers $d\ge1$ and $m\ge1$, and set
\[
\beta_d:=1+(d-1)^2.
\]
Let $G=(I,J,M)$ be $C_4$-free and suppose every right vertex has degree at
most $d$. Fix a \light{} right vertex $j_\star\in J$. Set
$I_\star=N_G(j_\star)$, and assume that for every $i\in I_\star$ there is a
\reservoir{}
\[
\mathcal R_i\subseteq N^\dagger_G(i)\setminus\{j_\star\}
\]
with
\[
|\mathcal R_i|\ge m\beta_d.
\]
Then $G$ contains a \wcmrt{m} at $j_\star$
with $\mathcal K_i\subseteq\mathcal R_i$ for every $i\in I_\star$.
\end{cor}

\begin{proof}
By Lemma~\ref{lem:reservoirs-produce-tree}, the \reservoir{}s contain subfamilies
$\mathcal K_i\subseteq\mathcal R_i$, $i\in I_\star$, which form an
\mrt{m} at $j_\star$. For every $i\in I_\star$ and every
$j\in\mathcal K_i$, we have $j\in N^\dagger_G(i)$, while
$j_\star\notin N^\dagger_G(i)$ because $j_\star$ is \light{}. Therefore the
defining property of $N^\dagger_G(i)$ gives
\[
|M_{ij}|\ge |M_{i,j_\star}|.
\]
This is exactly the weight-compatibility condition.
\end{proof}

The second ingredient combines sparse right-pair overlap with the light-root
deletion. It produces many candidate right roots whose left neighborhoods are
disjoint, whose neighboring left vertices have large degree, and which are
\light{} relative to the heavy-neighbor sets.

\begin{lemma}[Sparse overlap yields many light candidate roots]
\label{lem:many-candidate-roots}

Let $G=(I,J,M)$ have $|I|=k$ and $|J|=n$ satisfying $ k/n \le 1/32$
, and suppose every right degree
is at most $d$. For $i\in I$, write $D_i:=|N_G(i)|$, and let
$\eta:=\eta(G)$ be the \overlapparameter{} of $G$. Let $K\ge1$ satisfy
\[
K(8\eta+2/n)\le1.
\]
Then there exist distinct right vertices $j_1,\dots,j_K\in J$ such that:

\begin{enumerate}
\item the neighborhoods $I_s:=N_G(j_s)$ are pairwise disjoint;
\item for every $s\in[K]$ and every $i\in I_s$,
\(
D_i\ge \frac{n}{8k}.
\)
\item every $j_s$ is \light{}.

\end{enumerate}

\end{lemma}

\begin{proof}
We will define three undesirable sets of right vertices. 

First, let
\[
I_{\rm bad}:=\{i\in I:D_i<n/(8k)\}.
\]
Then
\[
\sum_{i\in I_{\rm bad}}D_i\le k\cdot \frac{n}{8k}=\frac n8.
\]
Hence at most $n/8$ right vertices touch $I_{\rm bad}$. Let $\mathcal V$ be the set of right vertices $j$ such that every $i\in N_G(j)$ satisfies $D_i\ge n/(8k)$. Then
\(
|\mathcal V|\ge \frac78 n.
\)

For a right vertex $j$, let
\[
h(j):=
\bigl|\{j'\in J\setminus\{j\}:N_G(j')\cap N_G(j)\ne\emptyset\}\bigr|.
\]
Since $\eta$ is the overlap probability for an ordered pair of distinct right vertices,
\[
\frac1n\sum_{j\in J}h(j)=\eta(n-1)\le \eta n.
\]
Let
\[
\mathcal L:=\{j\in J:h(j)\le 4\eta n\}.
\]
If $\eta=0$, then $\mathcal L=J$. If $\eta>0$, Markov's inequality gives $|J\setminus\mathcal L|\le n/4$.

Let
\[
\mathcal H
:=
\{j\in J:\ j\text{ is not \light{}}\}.
\]
Then every vertex in $\mathcal H$ belongs to $N^\dagger_G(i)$ for at least
one $i\in I$. Therefore
\[
|\mathcal H|
\le
\sum_{i\in I}|N^\dagger_G(i)|
=
\sum_{i\in I}\left\lceil\frac{D_i}{32d}\right\rceil
\le
\frac1{32d}\sum_{i\in I}D_i+k
\le
\left(\frac1{32}+\frac{k}{n}\right)n
\le
\frac n{16}.
\]
Set
\[
\mathcal U:=\mathcal V\cap\mathcal L\cap(J\setminus\mathcal H).
\]
Our assumptions imply
\[
|\mathcal U|\ge \frac9{16}n.
\]

Choose roots greedily from $\mathcal U$. Suppose $j_1,\dots,j_{s-1}$ have already been chosen. Each previous root $j_t$ forbids itself and at most $h(j_t)\le4\eta n$ right vertices whose neighborhoods meet $N_G(j_t)$. Thus the number of forbidden vertices is at most
\[
(s-1)(4\eta n+1)\le K(4\eta n+1).
\]
The hypothesis gives
\[
K(4\eta n+1)\le \frac n2.
\]
Since $|\mathcal U|\ge9n/16$, some vertex $j_s\in\mathcal U$ remains available. By construction, $j_s$ is distinct from the earlier roots and $N_G(j_s)$ is disjoint from every earlier root neighborhood. Repeating this $K$ times gives $j_1,\dots,j_K$ with pairwise disjoint neighborhoods.

Finally, every chosen root lies in $\mathcal V$, so every left neighbor of every chosen root has degree at least $n/(8k)$.
Every chosen root also lies outside $\mathcal H$, so it is \light{}.

\end{proof}

The third ingredient is purely probabilistic: among many disjoint witness
blocks, a uniform right-vertex sample captures one complete block with high
probability.

\begin{lemma}[Capturing one of many disjoint witness blocks]
\label{lem:capture-disjoint-blocks}

Fix integers $B\ge1$ and $K\ge1$, and real $\lambda>2$. Let $r$ be a positive integer and $n_0=\lceil\lambda r\rceil$. Assume
\[
(8\lambda)^{2B}\le K,
\qquad
K\log K\le \frac r2.
\]
Let $J$ be a set of size $n_0$ and assume there are pairwise disjoint subsets
\[
\mathcal W_1,\dots,\mathcal W_K\subseteq J,
\]
each of size at most $B$. If $\Omega\subseteq J$ is chosen uniformly from $\binom Jr$, then
\[
\mathbb P\big(\forall s\in[K]:\mathcal W_s\not\subseteq\Omega\big)
\le
\exp(-\sqrt K).
\]

The condition $K\log K\le r/2$ keeps the sequential exposure of the $K$ blocks from depleting the sample. The condition $(8\lambda)^{2B}\le K$ ensures that the capture probability of one block, repeated $K$ times, gives exponent at least $\sqrt K$.

\end{lemma}

\begin{remark}
In this paper we repeatedly estimate the probability that a uniformly sampled
set captures a fixed subset. If $W\subseteq J$ has size $b$ and $\Omega$ is
chosen uniformly from $\binom Jr$, then
\[
\mathbb P(W\subseteq\Omega)
=
\frac{\binom{|J|-b}{r-b}}{\binom{|J|}{r}}
=
\frac{\ff{r}{b}}{\ff{|J|}{b}},
\]
where, for integers $a,q\ge0$, we write
\begin{equation}
\label{eq:falling-factorial}
(a)_q:=a(a-1)\cdots(a-q+1)
\end{equation}
for the falling factorial, with the convention $(a)_0=1$.
\end{remark}

\begin{proof}
For each $s\in[K]$, let $\Gamma_s$ be the event $\mathcal W_s\subseteq\Omega$, and put $b_s:=|\mathcal W_s|\le B$.

We first record the scale bounds used below. From $(8\lambda)^{2B}\le K$,
\[
2B\log(8\lambda)\le \log K.
\]
Since $\lambda>2$, we have $\log(8\lambda)>1$, and hence
\[
B\le \log K.
\]
Also
\[
K\ge (8\lambda)^{2B}\ge 16^2>2.
\]
Hence
\[
KB\le K\log K\le \frac r2.
\]
Also, since $K\log K\le r/2$ and $\log K\ge1$, we have $K\le r/2$.

Reveal $\Omega$ block by block in the order $\mathcal W_1,\dots,\mathcal W_K$. After the first $s-1$ blocks have been examined, let $a_{s-1}$ be the number of sampled elements already found in $\mathcal W_1\cup\cdots\cup\mathcal W_{s-1}$. Conditional on the past, the restriction of $\Omega$ to the complement of these blocks is a uniform subset of size
\[
q_s:=r-a_{s-1}
\]
inside a ground set of size
\[
N_s:=n_0-|\mathcal W_1\cup\cdots\cup\mathcal W_{s-1}|.
\]
The scale bounds give
\[
q_s\ge r-KB\ge \frac r2
\]
and
\[
N_s\le n_0\le2\lambda r.
\]
Moreover,
\[
q_s-b_s
\ge
r-KB-B
\ge
\frac r2-\log K
\ge
\frac r4,
\]
because $K\le r/2$ implies $\log K\le r/4$ for $K\ge2$.

Therefore, conditional on the past,
\[
\mathbb P(\Gamma_s\mid\text{past})
=
\frac{\binom{N_s-b_s}{q_s-b_s}}{\binom{N_s}{q_s}}
=
\frac{q_s!\,(N_s-b_s)!}{(q_s-b_s)!\,N_s!}
=
\prod_{\ell=0}^{b_s-1}\frac{q_s-\ell}{N_s-\ell}
\ge
\left(\frac{q_s-b_s}{N_s}\right)^{b_s}
\ge
\left(\frac1{8\lambda}\right)^{b_s}
\ge
\left(\frac1{8\lambda}\right)^B.
\]

Set $p_\star:=(8\lambda)^{-B}$. Then, by sequential conditioning,
\[
\mathbb P\left(\bigcap_{s=1}^K\Gamma_s^c\right)
\le
(1-p_\star)^K
\le
\exp(-p_\star K).
\]

The hypothesis $(8\lambda)^{2B}\le K$ gives $p_\star K=K/(8\lambda)^B\ge\sqrt K$. Thus
\[
\mathbb P\left(\bigcap_{s=1}^K\Gamma_s^c\right)
\le
\exp(-\sqrt K).
\]

\end{proof}

\subsection{The Constant-Ratio \texorpdfstring{$C_4$}{C4}-free Case}

The next proposition combines the three ingredients when the ambient right side has size comparable to the final sample size.

\begin{prop}[Constant-ratio $C_4$-free sampling theorem]
\label{prop:c4-constant-ratio-subgraph}

Fix an integer $d\ge2$, a real $\varepsilon\in(0,1]$, and a number
$\eta\ge0$. Set
\[
\beta_d:=1+(d-1)^2.
\]
Let $k,r\ge1$ satisfy
\[
\varepsilon k\le r\le k.
\]
Let
\[
K:=
\max\left\{
\ell\in\mathbb N_{\ge1}:
\ell\log \ell\le \frac r2
\quad\text{and}\quad
\ell\le \frac1{16\eta}
\right\},
\]
where the second condition is omitted when $\eta=0$. Assume this set is nonempty.

Let $\lambda>2$, and define
\[
m:=\left\lfloor \frac{\varepsilon\lambda}{512d^2\beta_d}\right\rfloor,
\qquad
L:=m\beta_d,
\qquad
B:=1+dL.
\]
Assume
\[
\varepsilon\lambda\ge512d^2\beta_d,
\qquad
(8\lambda)^{2B}\le K.
\]
Let $n:=\lceil\lambda r\rceil$, and let $G=(I,J,M)$ be $C_4$-free with
\[
|I|=k,\qquad |J|=n,
\]
and right degrees satisfying
\[
2\le |N_G(j)|\le d
\qquad\forall j\in J.
\]
Assume
\[
\eta(G)\le\eta.
\]
If $\Omega$ is chosen uniformly from $\binom Jr$, then
\[
\mathbb P\left(
\text{$G[\Omega]$ contains a \wcmrt{m}}
\right)
\ge 1-\exp(-\sqrt K).
\]

\end{prop}

\begin{remark}[Scale of the admissibility condition]
Since
\[
B=1+d\beta_d m\asymp d^3m
\quad
\mbox{and}
\quad
m\asymp \frac{\varepsilon\lambda}{d^4},
\]
the exact condition
\[
(8\lambda)^{2B}\le K
\]
needed in Lemma~\ref{lem:capture-disjoint-blocks} should be read heuristically as
\[
\left(\frac{Cd^4m}{\varepsilon}\right)^{Cd^3m}\le K
\]
for a universal constant $C$. 
\end{remark}

\begin{proof}

\emph{Step 1: choose candidate roots.} For $i\in I$, write $D_i:=|N_G(i)|$. Lemma~\ref{lem:many-candidate-roots} requires
\[
K(8\eta(G)+2/n)\le1.
\]
The assumed overlap bound gives $8K\eta(G)\le8K\eta\le1/2$ when $\eta>0$, while $\eta(G)=0$ when $\eta=0$. Also $(8\lambda)^{2B}\le K$ implies $\log K\ge1$, so $K\log K\le r/2$ gives $K\le r/2$. Since $n=\lceil\lambda r\rceil$ and $\lambda>2$, we get $2K/n\le1/2$. Moreover,
\[
\frac{k}{n}
\le
\frac1{\varepsilon\lambda}
\le
\frac1{32}.
\]
Hence Lemma~\ref{lem:many-candidate-roots} applies.

We obtain $K$ distinct right vertices $j_1,\dots,j_K$ such that their neighborhoods
\[
I_s:=N_G(j_s)
\]
are pairwise disjoint and every $i\in I_s$ satisfies
\[
D_i\ge \frac n{8k}.
\]
Moreover, every $j_s$ is \light{}.

\emph{Step 2: build disjoint heavy reservoirs.} Let
\[
I':=\bigsqcup_{s=1}^K I_s,
\qquad
J':=J\setminus\{j_1,\dots,j_K\}.
\]
For $i\in I'$, let $s_i$ be the unique index with $i\in I_{s_i}$. Since
$j_{s_i}$ is \light{} and the root neighborhoods are pairwise disjoint,
\[
N^\dagger_G(i)\subseteq J'
\qquad
\forall i\in I'.
\]
Let $H$ be the bipartite graph on $(I',J')$ whose edges are
\[
i\sim_H j
\qquad\Longleftrightarrow\qquad
j\in N^\dagger_G(i).
\]
The right degrees in $H$ are at most $d$. For $i\in I'$, the left degree in
$H$ is $|N^\dagger_G(i)|=\left\lceil D_i/(32d)\right\rceil$, and hence
\[
\deg_H(i)\ge \frac{D_i}{32d}\ge \frac n{256dk}.
\]

The degree estimate shows that any vertex in $I'$ has a large neighborhood in the heavy-neighbor graph $H$, but these neighborhoods can overlap. 
We need to select, for each left vertex $i\in I'$, a subset
$\mathcal R_i$ of its neighbors in  $H$ of cardinality
$L$, with all $\mathcal R_i$ pairwise disjoint. This is a $b$-matching
problem: the left vertices have demand $L$, while the right vertices have capacity
$1$.
If $L$ were $1$, this would have been the standard perfect matching problem whose solvability is described by Hall's Matching  theorem.
The standard cloning reduction  reduces the problem to a
perfect matching by replacing each left vertex by $L$
copies with the same neighborhood. Application of Hall's Matching theorem to the cloned graph shows that  a required multi-matching exists if and only if
for every subset $T\subseteq I'$,
\[
|N_H(T)|\ge L|T|
\]
(see, for example, \cite{HalmosVaughan1950,Schrijver2003}). We verify this
Hall-type condition. For any $T\subseteq I'$, double-counting edges between
$T$ and $N_H(T)$ gives
\[
\frac n{256dk}|T|\le d\,|N_H(T)|.
\]
Therefore
\[
|N_H(T)|
\ge \frac{n}{256d^2k}|T|
\ge \frac{\varepsilon\lambda}{256d^2}|T|
\ge L|T|,
\]
where the second inequality uses $n\ge\lambda r\ge\varepsilon\lambda k$, and
the last follows from
\[
L=m\beta_d\le\frac{\varepsilon\lambda}{512d^2}.
\]
Thus Hall's condition holds, and the disjoint heavy \reservoir{}s exist.

For each $s\in[K]$, define the \witnessblock{}
\[
\mathcal W_s:=\{j_s\}\cup\bigcup_{i\in I_s}\mathcal R_i.
\]
The blocks are pairwise disjoint, and each has size at most
\[
1+dL=B.
\]

\emph{Step 3: capture a witness block.} The $K$ disjoint \witnessblock{}s of size at most $B$, together with the hypotheses $K\log K\le r/2$ and $(8\lambda)^{2B}\le K$, are exactly the input to Lemma~\ref{lem:capture-disjoint-blocks}. Hence
\[
\mathbb P\big(\exists s\in[K]:\mathcal W_s\subseteq\Omega\big)
\ge
1-\exp(-\sqrt K).
\]
On this event, fix $s$ with $\mathcal W_s\subseteq\Omega$. Inside $G[\Omega]$, each $i\in I_s$ has a \reservoir{}
\[
\mathcal R_i\subseteq N^\dagger_G(i)\setminus\{j_s\}
\]
of size $L=m\beta_d$. Since $j_s$ is \light{},
Corollary~\ref{cor:heavy-reservoirs-weight-compatible} gives a \wcmrt{m} at
$j_s$. All right vertices used by this tree lie in
$\Omega$, so the tree is contained in $G[\Omega]$.

\end{proof}

\subsection{The Finite \texorpdfstring{$C_4$}{C4}-free Case}

The previous proposition assumes that the ambient right side has size $\lceil\lambda r\rceil$ and that the overlap parameter is explicitly small. The next theorem removes these restrictions for an arbitrary large $C_4$-free right side. The proof first thins the right side to an intermediate set $J_0$ of size $\lceil\lambda_\star r\rceil$. Uniform thinning preserves the expected overlap parameter, and the bound on the \overlapparameter{} for $C_4$-free graphs makes the bad thinning event negligible.

\begin{lemma}[Uniform thinning preserves sparse overlap]
\label{lem:thinning-preserves-overlap}

Let $G=(I,J,M)$ be finite with $|J|=n\ge2$. Choose $J_0\subseteq J$ uniformly from $\binom J{n_0}$, where $2\le n_0\le n$, and let $G_0:=G[J_0]$. Then
\[
\mathbb E\,\eta(G_0)=\eta(G).
\]
Consequently, for every $\alpha>0$,
\[
\mathbb P\big(\eta(G_0)>\alpha\big)\le\frac{\eta(G)}{\alpha}.
\]

\end{lemma}

\begin{proof}
Conditional on $J_0$, the overlap parameter $\eta(G_0)$ is the probability that a uniformly chosen ordered pair of distinct right vertices from $J_0$ overlaps. If we first choose $J_0$ uniformly and then choose such an ordered pair from $J_0$, the resulting ordered pair is uniform over all distinct ordered pairs in $J$. Therefore the expected overlap parameter is exactly $\eta(G)$. The tail bound is Markov's inequality.

\end{proof}

\begin{theor}[Finite $C_4$-free sampling theorem]
\label{thm:c4-finite-subgraph}

There exists a universal constant
$C_{\tref{thm:c4-finite-subgraph}}\ge1$ such that the following holds.
Fix an integer $d\ge2$, a real $\rho\in(0,1]$, and positive integers $k,r,n$ with  
 $$2 \le \rho k \le r\le k\le n.$$

Let $G=(I,J,M)$ be $C_4$-free with $|I|=k$, $|J|=n$, and right degrees satisfying
\[
2\le |N_G(j)|\le d
\qquad\forall j\in J.
\]

Define
\[
K_\star:=
\max\left\{
K\in\mathbb N_{\ge1}:
K\log K\le \Lambda 
\right\}, \quad \mbox{ where } \Lambda :=\min\left(\frac r2,\frac nk\right). 
\]
and let $m_\star$ be a positive integer satisfying
\[
\left(\frac{C_{\tref{thm:c4-finite-subgraph}}d^4m_\star}{\rho}\right)^{C_{\tref{thm:c4-finite-subgraph}}d^3m_\star}
\le K_\star.
\]
Assume such an $m_\star$ exists.
If $\Omega\subseteq J$ is chosen uniformly from $\binom Jr$, then
\[
\mathbb P\bigl(\text{$G[\Omega]$ contains a \wcmrt{m_\star}}\bigr)
\ge
1-\exp(-\sqrt{K_\star})-\frac{32d}{\log K_\star}.
\]

\end{theor}

\begin{remark}
\label{rem:c4-finite-admissible-scale}
 The constant $C_{\tref{thm:c4-finite-subgraph}}$ is chosen only so that this simpler displayed condition implies the exact admissibility condition
\[
(8\lambda_\star)^{2B_\star}\le K_\star
\]
needed in Proposition~\ref{prop:c4-constant-ratio-subgraph}, for the $\lambda_\star$ and $B_\star$ defined in the proof below.
Thus $K_\star$ is the number of witness-block attempts available after accounting for both the final sample size $r$ and the ambient ratio $\theta=n/k$, while $m_\star$ is the largest tree width whose witness blocks can still be captured with these attempts.

After increasing $C_{\tref{thm:c4-finite-subgraph}}$ if necessary, the existence of $m_\star$ guarantees that both $K_\star$ and $\Lambda$ are greater than some large universal constants, so that taking logarithms twice is valid and remains positive.

Taking logarithms, the admissibility condition for an integer $m$ is equivalent to
\[
C_{\tref{thm:c4-finite-subgraph}}d^3m
\,
\log\left(\frac{C_{\tref{thm:c4-finite-subgraph}}d^4m}{\rho}\right)
\le \log K_\star.
\]
For $\Lambda$ large, the definition of $K_\star$ gives
\[
\log K_\star\asymp \log\Lambda.
\]
Set $x=C_{\tref{thm:c4-finite-subgraph}}d^3m$ and $y=\log K_\star$. The boundary equation associated with the logarithmic admissibility condition is
\[
y = x\log\left(x\frac d\rho\right).
\]
For $x$ greater than a sufficiently large universal constant, we have
\[
\log\left(\frac d\rho y\right)
=
\log\left(\frac d\rho x\log\left(x\frac d\rho\right)\right)
\asymp
\log\left(\frac d\rho x\right),
\]
which in turn implies that
\[
x \asymp \frac{y}{\log\left(\frac d\rho y\right)}.
\]
Thus, the largest admissible integer $m_\star$ satisfies
\[
m_\star
\asymp
\frac{\log\Lambda}{d^3\log(\tfrac{d}{\rho}\log\Lambda)}.
\]

\end{remark}

\begin{proof}

\emph{Step 1: choose the intermediate scale.}
We start with selecting the intermediate sample size $n_0= \lceil \lambda_* n \rceil$ which will allow us to apply Proposition~\ref{prop:c4-constant-ratio-subgraph} and checking that the assumptions on $K=K_\star$ appearing in this proposition are satisfied.
 Set
\[
\lambda_\star:=\frac{512d^2\beta_d m_\star}{\rho},
\qquad
B_\star:=1+d\beta_d m_\star
\qquad 
\theta:=\frac nk.
\]
By choosing the universal constant $C_{\tref{thm:c4-finite-subgraph}}$
sufficiently large, the defining inequality for $m_\star$, together with
$\beta_d\le d^2$, $\rho\le1$, and $m_\star\ge1$, implies
\[
(8\lambda_\star)^{2B_\star}\le K_\star,
\qquad
8\rho\lambda_\star\le K_\star,
\qquad
\rho\lambda_\star\ge512d^2\beta_d.
\]
Indeed, $m_\star\ge1$ and $\rho\le1$ give
$K_\star\ge (C_{\tref{thm:c4-finite-subgraph}}d^4)^{C_{\tref{thm:c4-finite-subgraph}}d^3}$.
After increasing $C_{\tref{thm:c4-finite-subgraph}}$ if necessary, the
existence of $m_\star$ therefore implies $K_\star\ge3$ and
\[
\theta
\ge
K_\star\log K_\star
\ge
32d\beta_d.
\]
Since $K_\star\log K_\star\le \theta$, the second inequality gives
\[
\lambda_\star r=\rho\lambda_\star k
\le
\frac18K_\star k
\le
\frac18\theta k
=
\frac n8.
\]
With
\[
n_0:=\lceil\lambda_\star r\rceil
,
\]
the bound $\lambda_\star r\le n/8$ gives
\[
n_0
\le
\lambda_\star r+1
\le
\frac n8+1
\le n,
\]
where the last inequality follows from $n \ge 2$. 

\emph{Step 2: control overlap after thinning.} Choose $J_0\subseteq J$ uniformly from $\binom J{n_0}$ and set $G_0:=G[J_0]$. Then $G_0$ is $C_4$-free and its right degrees still lie between $2$ and $d$. Let
\[
\mathcal E:=\left\{\eta(G_0)\le\frac1{16K_\star}\right\}.
\]
By Lemma~\ref{lem:thinning-preserves-overlap},
\[
\mathbb P(\mathcal E^c)\le16K_\star\eta(G).
\]

\emph{Step 3: apply the constant-ratio theorem.} Condition on any realization of $J_0$ in $\mathcal E$. Also $\lambda_\star>2$, since
\[
\lambda_\star
=
\frac{512d^2\beta_d m_\star}{\rho}
\ge
512d^2\beta_d
>2,
\]
where $\rho\le1$ and $m_\star\ge1$. Proposition~\ref{prop:c4-constant-ratio-subgraph} applies to $G_0$ with $\varepsilon=\rho$, $\lambda=\lambda_\star$, $\eta=1/(16K_\star)$, and $K=K_\star$: the overlap condition follows from $\mathcal E$, the bound $K_\star\log K_\star\le r/2$ follows from the definition of $K_\star$, and the remaining scale conditions were checked in Step 1. Moreover
\[
\left\lfloor\frac{\rho\lambda_\star}{512d^2\beta_d}\right\rfloor=m_\star.
\]
Thus, if $\Omega$ is chosen uniformly from $\binom{J_0}r$, then
\[
\mathbb P\left(
\begin{gathered}
\text{$G_0[\Omega]$ contains}\\
\text{a \wcmrt{m_\star}}
\end{gathered}
\ \middle|\ J_0\right)
\ge
1-\exp(-\sqrt{K_\star}).
\]

\emph{Step 4: return to the original right side.} For the nested sampling procedure,
\[
\mathbb P\left(
\begin{gathered}
\text{$G[\Omega]$ fails to contain}\\
\text{a \wcmrt{m_\star}}
\end{gathered}
\right)
\le
\mathbb P(\mathcal E^c)+\exp(-\sqrt{K_\star})
\le
16K_\star\eta(G)+\exp(-\sqrt{K_\star}).
\]
By Lemma~\ref{lem:c4-sparse-overlap},
\[
\eta(G)\le2d\,\frac{k}{n}=\frac{2d}{\theta}.
\]
Hence
\[
16K_\star\eta(G)
\le
32d\,\frac{K_\star}{\theta}
\le
\frac{32d}{\log K_\star},
\]
where the last inequality uses $K_\star\log K_\star\le \theta$.

Finally, nested uniform sampling gives a uniform $r$-subset of $J$: if first $J_0$ is uniform in $\binom J{n_0}$ and then $\Omega$ is uniform in $\binom{J_0}r$, then $\Omega$ is uniform in $\binom Jr$. This proves the stated bound.

\end{proof}

\section{Trichotomy and Sampling}
\label{sec:bridge-sampling}

This section reduces \lsfci{} to the $C_4$-free sampling theorem proved in Section~\ref{sec:c4-sampling-subgraph}. The reduction has two parts. First, we choose an intermediate right set and delete the few sampled vertices that participate in sampled four-cycles; the remaining good set is $C_4$-free and still has asymptotically full size. Second, a transfer lemma says that a uniform sample from the original graph contains the same matching rooted tree patterns as a uniform sample from the good $C_4$-free subgraph, up to a negligible loss.

\subsection{From Local Four-Cycle Sparsity to a Good \texorpdfstring{$C_4$}{C4}-free Subgraph}
\begin{definition}[Double-overlap]
Column $j'$ of a matrix $M$ 
is a \keyterm{double-overlap partner} of $j$ if
\begin{equation}
\label{def:double-overlap-partner}
\bigl|\operatorname{supp}M_{\cdot j}
\cap
\operatorname{supp}M_{\cdot j'}\bigr|\ge2.
\end{equation}

\label{ass:local-sparse-double-overlap}
For a deterministic sequence of matrices  
\(
M^{(k)} \in \mathbb{R}^{k \times n_k}
\)
for $k \in \mathbb{N}$, 
we say that it satisfies
the \keyterm{local sparse double-overlap condition} if
there are sets
\[
\mathcal G_k\subseteq[n_k]
\quad \mbox{ and } \quad \varepsilon_k \searrow 0
\]
such that
\[
|\mathcal G_k|=(1-o(1))n_k
\]
and, for every $j\in\mathcal G_k$,
\[
\left|
\left\{
j'\in[n_k]\setminus\{j\}:
\bigl|\operatorname{supp}M^{(k)}_{\cdot j}
\cap
\operatorname{supp}M^{(k)}_{\cdot j'}\bigr|\ge2
\right\}
\right|
\le
\varepsilon_k\,\frac{n_k}k.
\]
\end{definition}

Observe that the assumption of the main theorem that the number of pairs of columns overlapping on two
or more rows is of order $o(n_k^2/k)$, immediately implies \lsdoubleoverlap{} for some sequence of numbers $\varepsilon_k\to0$. 

In graph language, if a pair of columns $(j,j')$ are \doubleoverlappartners, then the corresponding right vertices in the support graph have at least two common left neighbors. In other words, they form a four-cycle $C_4$ in the support graph. So from the graph-theoretic perspective, we call it the \lsfci{} condition:

\begin{definition}[Local sparse four-cycle incidence]
\label{def:local-sparse-four-cycle-incidence}

Let $G^{(k)}=(I^{(k)},J^{(k)},E^{(k)})$ be a sequence of bipartite graphs with
\[
|I^{(k)}|=k,\qquad |J^{(k)}|=n_k.
\]
For right vertices $j,j'\in J^{(k)}$, we say that $j'$ is a \fcp{} of $j$ if
\[
j'\ne j
\qquad\text{and}\qquad
|N_{G^{(k)}}(j)\cap N_{G^{(k)}}(j')|\ge2.
\]
Equivalently, the two right vertices lie together in at least one $4$-cycle.

The sequence $G^{(k)}$ has \keyterm{local sparse four-cycle incidence} if there are numbers $\varepsilon_k\to0$ such that the bad set
\[
B_{\mathrm{fc}}^{(k)}
:=
\Big\{ j\in J^{(k)}:\ j\text{ has more than }\varepsilon_k\frac{n_k}{k}\,\text{\fcp{}s}\Big\}
\]
satisfies $|B_{\mathrm{fc}}^{(k)}|=o(n_k)$.
\end{definition}

Thus the condition is local in the right vertex $j$: it allows some exceptional right vertices, but a typical right vertex has very few \fcp{}s at the scale $n_k/k$.

\begin{remark}[Local versus global four-cycle count]
\label{rem: local-sparse-four-cycle-incidence}

A stronger global hypothesis would be
\[
\mbox{the number of }4\text{-cycles in }G^{(k)}
=
o\left(\frac{n_k^2}{k}\right).
\]
This global condition implies \lsfci{} by Markov's inequality.

\end{remark}

The next proposition shows that the \lsfci{} can be upgraded to $C_4$-freeness by passing to a subset of a large size in a uniformly chosen sample.
\begin{prop}[Producing a large $C_4$-free good subgraph]
\label{prop:produce-c4-good-subgraph}

Let $G^{(k)}=(I^{(k)},J^{(k)},E^{(k)})$ be a sequence of finite bipartite graphs with
\[
|I^{(k)}|=k,\qquad |J^{(k)}|=n_k,\qquad \frac{n_k}{k}\to\infty.
\]
Assume $G^{(k)}$ has \lsfci{}, with corresponding bad sets
$B_{\mathrm{fc}}^{(k)}\subseteq J^{(k)}$ and numbers $\varepsilon_k\to0$.
There is an absolute constant
$c_{\tref{prop:produce-c4-good-subgraph}}>0$ with the following property.
Let
\[
\widetilde n_k
:=
\min\left\{
\left\lceil c_{\tref{prop:produce-c4-good-subgraph}}\varepsilon_k^{-1} k\right\rceil,
n_k
\right\},
\]
and let
$\widetilde J^{(k)}\subseteq J^{(k)}$ be uniform of size $\widetilde n_k$.
With probability $1-o(1)$, there is a subset
$\good{\widetilde J}^{(k)}\subseteq\widetilde J^{(k)}$ such that
\[
\good{\widetilde J}^{(k)}
\subseteq
\widetilde J^{(k)}\setminus B_{\mathrm{fc}}^{(k)},
\qquad
G^{(k)}[\good{\widetilde J}^{(k)}]\text{ is }C_4\text{-free}
\qquad\text{and}\qquad
|\good{\widetilde J}^{(k)}|\ge\frac9{10}\widetilde n_k.
\]

\end{prop}

\begin{proof}
Let
\[
c:=c_{\tref{prop:produce-c4-good-subgraph}},
\qquad
n:=n_k,\qquad
\varepsilon:=\varepsilon_k,\qquad
B:=B_{\mathrm{fc}}^{(k)}.
\]
Set $m:=\widetilde n_k$ and choose $\widetilde J:=\widetilde J^{(k)}$
uniformly from $\binom{J^{(k)}}{m}$. Since
$m\le c\varepsilon^{-1}k+1$, we have
\[
\varepsilon\frac{m}{k}\le c+o(1).
\]

First remove the four-cycle-bad vertices. Let
\(
Z:=|\widetilde J\cap B|.
\)
Then
\[
\mathbb E Z
=
m\,\frac{|B|}{n}
=o(m),
\]
so $Z=o(m)$ with probability $1-o(1)$ by Markov's inequality.

Let $H$ be the graph on $J^{(k)}\setminus B$ in which two vertices are
adjacent if they are \fcp{}s. By the definition of $B$,
\[
\Delta(H)\le\varepsilon\frac nk,
\qquad
|E(H)|\le \frac12 n\varepsilon\frac nk,
\]
where $\Delta(H)$ denotes the maximal degree, and $E(H)$ is the set of edges of $H$.
Let $W$ be the number of edges of $H$ contained in $\widetilde J$. Then
\[
\mathbb E W
=
|E(H)|\frac{\ff{m}{2}}{\ff{n}{2}}
\le
C\varepsilon\frac{m^2}{k}
\le
C(c+o(1))m,
\]
where we recall that $\ff{a}{q}:=a(a-1)\cdots(a-q+1)$ is the falling factorial (see \eqref{eq:falling-factorial}).
Choosing $c_{\tref{prop:produce-c4-good-subgraph}}$ sufficiently small, we
may assume that $\mathbb E W\le m/100$ for all sufficiently large $k$.

We next bound the variance of $W$. Write
$W=\sum_{e\in E(H)}X_e$, where $X_e$ is the indicator of the event
$e\subseteq\widetilde J$. If $e$ and $f$ are disjoint edges, then sampling
without replacement gives
\[
\mathbb E X_eX_f
=
\frac{\ff{m}{4}}{\ff{n}{4}}
\le
\left(\frac{\ff{m}{2}}{\ff{n}{2}}\right)^2
=
\mathbb E X_e\,\mathbb E X_f.
\]
Thus disjoint edge pairs have nonpositive covariance. Hence
\[
\operatorname{Var}(W)
\le
\mathbb E W
+
\sum_{\substack{e,f\in E(H):\,e\ne f\\ e\cap f\ne\emptyset}}
\mathbb E X_eX_f.
\]
The number of ordered pairs of distinct adjacent edges in $H$ is at most
$n\Delta(H)^2$, and for each such pair
\[
\mathbb E X_eX_f
=
\frac{\ff{m}{3}}{\ff{n}{3}}
\le
C\left(\frac mn\right)^3.
\]
Therefore
\[
\operatorname{Var}(W)
\le
C m
+
C n\left(\varepsilon\frac nk\right)^2\left(\frac mn\right)^3
\le
C m+C\varepsilon^2\frac{m^3}{k^2}
=o(m^2),
\]
where the last step uses $m\to\infty$ and $m\le c\varepsilon^{-1}k+1$.
By Chebyshev's inequality,
\[
\mathbb P\left(W>\frac1{20}m\right)=o(1).
\]

On the event $Z=o(m)$ and $W\le m/20$, remove from $\widetilde J$ all
vertices in $B$ and then delete one endpoint from each remaining edge of
$H[\widetilde J\setminus B]$. The resulting set
$\good{\widetilde J}^{(k)}$ is contained in
$\widetilde J^{(k)}\setminus B_{\mathrm{fc}}^{(k)}$, has size at least
$m-o(m)-m/20\ge 9m/10$ for all large $k$, and is independent in $H$. Hence no
two of its vertices are \fcp{}s, so
$G^{(k)}[\good{\widetilde J}^{(k)}]$ is $C_4$-free.

\end{proof}

\subsection{Average-Degree Trichotomy}
The next proposition is the main structural result.
It states that if a sequence of graphs satisfies the conditions of Theorem \ref{thm:weighted-local-sparse}, then, with high probability, a random sample contains either an empty column, or two columns with only one entry in each row, or  a \wcmrt{m_k} with $m_k \to \infty$.
\begin{prop}[Average-degree trichotomy]
\label{prop:average-degree-bridge-alternatives}

Fix $\rho\in(0,1]$ and a constant $d_{\rm av}>0$. Let
\[
G^{(k)}=(I^{(k)},J^{(k)},M^{(k)})
\]
be a deterministic sequence of weighted bipartite graphs, where
$M^{(k)}\in\mathbb R^{I^{(k)}\times J^{(k)}}$.
Assume
\[
|I^{(k)}|=k,\qquad |J^{(k)}|=n_k,\qquad \frac{n_k}{k}\to\infty.
\]
Assume the following properties hold for all sufficiently large $k$:
\begin{itemize}
\item The support graph of $G^{(k)}$ has \lsfci{} with parameter $\varepsilon_k\to0$.
\item The average right degree is at most $d_{\rm av}$:
\[
\frac1{n_k}\sum_{j\in J^{(k)}}
|N_{G^{(k)}}(j)|
\le d_{\rm av}.
\]
\item The support is nonempty, and the average magnitude of the nonzero weights is at most $1$:
\[
\frac1{|E^{(k)}|}\sum_{(i,j)\in E^{(k)}} |M^{(k)}_{ij}|
\le 1,
\]
where $E^{(k)} = \{(i,j)\in I^{(k)}\times J^{(k)}:M^{(k)}_{ij}\ne0\}$ is the edge set of $G^{(k)}$; see Definition \ref{def:weighted-bipartite-graph}.
\end{itemize}
Let $r(k)$ be an integer sequence with $\rho k\le r(k)\le k$.
Then there is a sequence $m_k\to\infty$ such that, for a uniform sample
\[
\Omega_k\subseteq J^{(k)},\qquad \text{of size } |\Omega_k|=r(k),
\]
with probability $1-o(1)$ at least one of the following holds:

\begin{enumerate}
\item $G^{(k)}[\Omega_k]$ contains an isolated right vertex.

\item $G^{(k)}[\Omega_k]$ contains distinct degree-one right vertices $j,j'$ with the same unique left neighbor:
\[
N_{G^{(k)}}(j)=N_{G^{(k)}}(j')=\{i\}
\]
for some $i\in I^{(k)}$.

\item $G^{(k)}[\Omega_k]$ contains a \wcmrt{m_k}
whose right vertices have support degree at most $40d_{\rm av}$ and whose
incident nonzero weights have absolute value at most $10d_{\rm av}$.

\end{enumerate}
\end{prop}

The proof uses the following deterministic bad right-vertex sets. The first
two are the low-degree obstruction sets. For a degree cutoff $\Delta\ge2$,
the third set is the high-degree bad set. For a weight cutoff $A>0$, the
fourth set is the large-entry bad set. After fixing a witness
$\varepsilon_k\to0$ for \lsfci{}, the fifth set is the four-cycle-bad set:
\begin{equation}
\label{eq:bridge-deterministic-sets}
\begin{aligned}
B_0^{(k)}
&:=\{j\in J^{(k)}:|N_{G^{(k)}}(j)|=0\},\\
B_1^{(k)}
&:=\{j\in J^{(k)}:|N_{G^{(k)}}(j)|=1\},\\
B_{\mathrm{deg}}^{(k)}(\Delta)
&:=\{j\in J^{(k)}:|N_{G^{(k)}}(j)|>\Delta\},\\
B_{\mathrm{amp}}^{(k)}(A)
&:=\{j\in J^{(k)}:\exists i\in N_{G^{(k)}}(j)\text{ with }|M^{(k)}_{ij}|>A\},\\
B_{\mathrm{fc}}^{(k)}
&:=\Big\{j\in J^{(k)}:\ j\text{ has more than }
\varepsilon_k\frac{n_k}k\text{ \fcp{}s}\Big\}.
\end{aligned}
\end{equation}

For all the above bad sets, the last three can be chosen to have relatively small sizes. The size of $B_0^{(k)}$ and $B_1^{(k)}$ can be large, but they create the first two alternatives in the conclusion of Proposition~\ref{prop:average-degree-bridge-alternatives}, as shown in the next two lemmas. 

\begin{lemma}[Many isolated right vertices]
\label{lem:many-isolated-right-vertices}
Work under the assumptions and notation of Proposition~\ref{prop:average-degree-bridge-alternatives}. Let $\Omega_k\subseteq J^{(k)}$ be uniform of size $r(k)$, and let $\delta_k\downarrow0$ satisfy $\delta_k k\to\infty$.
If $|B_0^{(k)}|\ge\delta_k n_k$, then $G^{(k)}[\Omega_k]$ contains an isolated right vertex with probability $1-o(1)$.
\end{lemma}

\begin{proof}
If $|B_0^{(k)}| + r(k) > n_k$, then $\Omega_k$ must contain an isolated right vertex. Assume now that $|B_0^{(k)}| + r(k) \le n_k$. Then the probability that $\Omega_k$ contains no isolated right vertex is
\[
\mathbb P(\Omega_k\cap B_0^{(k)}=\emptyset)
=
\prod_{s=0}^{r(k)-1}\left(1-\frac{|B_0^{(k)}|}{n_k-s}\right)
\le
\left(1-\frac{|B_0^{(k)}|}{n_k}\right)^{r(k)}
\le
\exp\left(-\frac{r(k)|B_0^{(k)}|}{n_k}\right)
\le
\exp(-c\delta_k k)
=o(1),
\]
where $c>0$ depends only on $\rho$.
\end{proof}

For the next lemma, we need a stronger condition on $\delta_k$, instead of $\delta_kk \to \infty$, we require $\delta_k^2k \to \infty$.

\begin{lemma}[Many degree-one right vertices]
\label{lem:many-degree-one-right-vertices}
Work under the assumptions and notation of Proposition~\ref{prop:average-degree-bridge-alternatives}. Let $\Omega_k\subseteq J^{(k)}$ be uniform of size $r(k)$, and let $\delta_k\downarrow0$ satisfy
\[
\delta_k \frac{n_k}{k}\to\infty,
\qquad
\delta_k^2k\to\infty.
\]
If $|B_1^{(k)}|\ge\delta_k n_k$, then with probability $1-o(1)$, the graph $G^{(k)}[\Omega_k]$ contains distinct degree-one right vertices $j,j'$ with the same unique left neighbor.
\end{lemma}

\begin{proof}
For $i\in I^{(k)}$, let
\[
t_i:=|\{j\in B_1^{(k)}:N_{G^{(k)}}(j)=\{i\}\}|,
\qquad
T:=\sum_i t_i=|B_1^{(k)}|.
\]
Since $\delta_k R_k\to\infty$, we have $T\gg k$. Let
\[
S:=|\Omega_k\cap B_1^{(k)}|.
\]
Then
\[
\mathbb E S
=
r(k)\frac{T}{n_k}
\ge
\rho\delta_k k
\to\infty.
\]
By Hoeffding's comparison theorem for sampling without replacement
\cite[Theorem~4]{Hoeffding1963} (see also
\cite{BoucheronLugosiMassart2013} for a modern proof), the binomial Chernoff
lower-tail bound applies to the hypergeometric random variable $S$. Hence
\[
\mathbb P\left(S\le \frac12\mathbb E S\right)
\le
\exp(-c\mathbb E S)
=o(1).
\]
In particular, $S\ge c'\delta_k k$ with probability $1-o(1)$.

Conditional on $S=s$, the set $\Omega_k\cap B_1^{(k)}$ is a uniform $s$-subset of $B_1^{(k)}$. The no-collision event is the event that this subset contains at most one vertex from each block
\[
\{j\in B_1^{(k)}:N_{G^{(k)}}(j)=\{i\}\},
\qquad i\in I^{(k)}.
\]
Thus
\[
\mathbb P(\text{no collision}\mid S=s)
=
\frac{\sum_{A\subseteq I^{(k)},\,|A|=s}\prod_{i\in A}t_i}{\binom Ts},
\]
where the numerator counts the choices of $s$ blocks and one degree-one vertex from each chosen block. To bound this probability, note that the numerator is the $s$-th elementary symmetric sum of the numbers $(t_i:i\in I^{(k)})$. Under the constraint that $T = \sum_i t_i$ is fixed and $t_i\ge0$, the elementary symmetric sums are maximized when all $t_i$ are equal, which can be established by showing that averaging any two $t_i$ and applying the AM-GM inequality will only increase the symmetric sums. This is known as Maclaurin's inequality for elementary symmetric means \cite[p.~52]{HardyLittlewoodPolya1952}, which yields
\[
\sum_{A\subseteq I^{(k)},\,|A|=s}\prod_{i\in A}t_i
\le
\binom ks\left(\frac Tk\right)^s.
\]
The right-hand side is exactly the value when $t_i=T/k$ for all $i$.

An estimate of the conditional probability of the lack of collisions follows the standard birthday paradox argument.
More precisely,
\[
\mathbb P(\text{no collision}\mid S=s)
\le
\frac{\binom ks(T/k)^s}{\binom Ts}
=
\prod_{i=0}^{s-1}\frac{k-i}{k}
\prod_{i=0}^{s-1}\frac{T}{T-i},
\]
where
\[
\prod_{i=0}^{s-1} \frac{k-i}{k}
= \prod_{i=0}^{s-1} \left(1 - \frac{i}{k}\right)
\le \exp\left(-\sum_{i=0}^{s-1} \frac{i}{k}\right)
= \exp\left(-\frac{s(s-1)}{2k}\right).
\]
Our assumption on $\delta_k$ implies that $s=o(T)$, because $s\le r(k)\le k$, while
\[
\frac{T}{k}\ge \frac{\delta_k n_k}{k} \to \infty.
\]
Hence,
\[
\prod_{i=0}^{s-1} \frac{T}{T-i}
=
\prod_{i=0}^{s-1} \left( 1 + \frac{i}{T-i} \right)
\le
\exp\left(\sum_{i=0}^{s-1} 2\frac{i}{T}\right)
= \exp\left(\frac{s(s-1)}{T}\right).
\]

This shows that for every $c'\delta_k k \le s\le r(k)$,
\[
\mathbb P(\text{no collision}\mid S=s)
\le \exp \left(- s(s-1) \left[ \frac{1}{2k}- \frac{1}{T} \right] \right)
\le
\exp\left(-c''\frac{s^2}{k}\right)
\]
for some universal constant $c''>0$.

Therefore, on the likely event $S\ge c'\delta_k k$, the exponent is at least a constant times $\delta_k^2k\to\infty$. Thus a collision occurs with probability $1-o(1)$.
\end{proof}

\begin{lemma}[Large-entry right vertices]
\label{lem:large-entry-right-vertices}
Work under the assumptions and notation of Proposition~\ref{prop:average-degree-bridge-alternatives}. Then
\[
|B_{\mathrm{amp}}^{(k)}(10d_{\rm av})|
\le
\frac1{10}n_k.
\]
\end{lemma}

\begin{proof}
Recall that
$E^{(k)} = \{(i,j)\in I^{(k)}\times J^{(k)}:M^{(k)}_{ij}\ne0\}$.
For each $j\in B_{\mathrm{amp}}^{(k)}(10d_{\rm av})$, choose one edge
$(i_j,j)$ with $|M^{(k)}_{i_jj}|>10d_{\rm av}$. The chosen edges are
distinct. Hence
\[
10d_{\rm av}|B_{\mathrm{amp}}^{(k)}(10d_{\rm av})|
\le
\sum_{j\in B_{\mathrm{amp}}^{(k)}(10d_{\rm av})} |M^{(k)}_{i_jj}|
\le
\sum_{(i,j)\in E^{(k)}}|M^{(k)}_{ij}|
\le
|E^{(k)}|,
\]
where the last inequality follows from the average magnitude assumption of Proposition~\ref{prop:average-degree-bridge-alternatives}.
Further, the average right-degree bound implies
\[
|E^{(k)}|
=
\sum_{j\in J^{(k)}} |N_{G^{(k)}}(j)|
\le d_{\rm av}n_k.
\]
Combining the last two inequalities and dividing by $d_{\rm av}$ gives the desired bound.
\end{proof}

\begin{lemma}[Weighted subgraph after trimming]
\label{lem:weighted-subgraph-after-trimming}
Work under the assumptions and notation of Proposition~\ref{prop:average-degree-bridge-alternatives}. Let $\delta_k\downarrow0$ and assume
\[
|B_0^{(k)}|<\delta_k n_k,
\qquad
|B_1^{(k)}|<\delta_k n_k.
\]
Fix a witness $\varepsilon_k\to0$ for \lsfci{}.
Set
\[
\widetilde n_k
:=
\min\left\{
\left\lceil c_{\tref{prop:produce-c4-good-subgraph}}\varepsilon_k^{-1} k\right\rceil,
n_k
\right\}.
\]
Let
\[
\widetilde J^{(k)}\subseteq J^{(k)}
\]
be uniform with $|\widetilde J^{(k)}|=\widetilde n_k$.

Then, with probability $1-o(1)$, there is a set
$J_{\rm sub}^{(k)}\subseteq\widetilde J^{(k)}$ such that
\[
|J_{\rm sub}^{(k)}|\ge\frac23\widetilde n_k,
\]
the graph $G^{(k)}[J_{\rm sub}^{(k)}]$ is $C_4$-free, all right support
degrees in this subgraph lie between $2$ and $40d_{\rm av}$, and all its nonzero
weights have absolute value at most $10d_{\rm av}$.
\end{lemma}

\begin{remark}[Two-stage exposure]
\label{rem:bridge-two-stage-exposure}
The original uniform choice of $\Omega_k\in\binom{J^{(k)}}{r(k)}$ is realized equivalently in
two stages. First choose
$\widetilde J^{(k)}\in\binom{J^{(k)}}{\widetilde n_k}$ uniformly, and then
choose $\Omega_k\in\binom{\widetilde J^{(k)}}{r(k)}$ uniformly. 
\end{remark}

\begin{proof}
Since the support is nonempty and the average right support degree is at most
 $d_{\rm av}$, the case $d_{\rm av}=0$ is impossible.
Set
\[
B_{\mathrm{deg}}^{(k)}
:=
B_{\mathrm{deg}}^{(k)}(40d_{\rm av})
\quad\mbox{and}\quad
B_{\mathrm{amp}}^{(k)}
:=
B_{\mathrm{amp}}^{(k)}(10d_{\rm av}).
\]
Markov's inequality and the average right support-degree bound give
\[
\left|B_{\mathrm{deg}}^{(k)}\right|
\le
\frac1{40}n_k.
\]
By Lemma~\ref{lem:large-entry-right-vertices},
\[
|B_{\mathrm{amp}}^{(k)}|
\le
\frac1{10}n_k.
\]
The same witness $\varepsilon_k$ defines the four-cycle-bad set
$B_{\mathrm{fc}}^{(k)}$ in \eqref{eq:bridge-deterministic-sets}. Set
\[
B_{\mathrm{trim}}^{(k)}
:=
B_0^{(k)}\cup B_1^{(k)}\cup B_{\mathrm{deg}}^{(k)}
\cup B_{\mathrm{amp}}^{(k)}.
\]
The assumptions on $B_0^{(k)}$ and $B_1^{(k)}$, together with the degree bound
and the large-entry lemma, imply that for all sufficiently large $k$,
\[
|B_{\mathrm{trim}}^{(k)}|
\le\frac3{20} n_k.
\]

Applying Proposition~\ref{prop:produce-c4-good-subgraph} to the support graph of $G^{(k)}$,
with probability $1-o(1)$ there is a set
\[
J_{C_4}^{(k)}\subseteq\widetilde J^{(k)}
\]
such that
\[
|J_{C_4}^{(k)}|\ge\frac9{10}\widetilde n_k,
\qquad
G^{(k)}[J_{C_4}^{(k)}]\text{ is }C_4\text{-free}.
\]
Also, since $|B_{\mathrm{trim}}^{(k)}|\le 3n_k/20$, Hoeffding's comparison theorem for sampling without replacement \cite[Theorem~4]{Hoeffding1963} (see also \cite{BoucheronLugosiMassart2013}) and the binomial Chernoff bound give
\[
|\widetilde J^{(k)}\cap B_{\mathrm{trim}}^{(k)}|
\le
\frac15\widetilde n_k
\]
with probability $1-o(1)$.

On the intersection of these two high-probability events, define
\[
J_{\rm sub}^{(k)}
:=
J_{C_4}^{(k)}
\setminus
B_{\mathrm{trim}}^{(k)}.
\]
Then in fact
\[
|J_{\rm sub}^{(k)}|
\ge
\frac7{10}\widetilde n_k,
\]
and hence $|J_{\rm sub}^{(k)}|\ge(2/3)\widetilde n_k$ for all large $k$. The
remaining structural conclusions follow from the definitions of
$B_{\mathrm{trim}}^{(k)}$ and $J_{C_4}^{(k)}$.
\end{proof}

\begin{proof}[Proof of Proposition~\ref{prop:average-degree-bridge-alternatives}]
Choose $\delta_k\downarrow0$ so that
\[
\delta_k \frac{n_k}{k}\to\infty,
\qquad
\delta_k^2 k\to\infty.
\]
For instance, one can take $\delta_k=\min \left(\frac{n_k}{k},\sqrt k\right)^{-1/2}$.

If $|B_0^{(k)}|\ge\delta_k n_k$, then Lemma~\ref{lem:many-isolated-right-vertices} gives the first case with probability $1-o(1)$. If $|B_1^{(k)}|\ge\delta_k n_k$, then Lemma~\ref{lem:many-degree-one-right-vertices} gives the second case with probability $1-o(1)$.
It remains to consider the complementary case
\[
|B_0^{(k)}|<\delta_k n_k,
\qquad
|B_1^{(k)}|<\delta_k n_k.
\]
In particular, this case implies
\begin{equation}
  \label{eq:average-degree-trichotomy-40d}
40d_{\rm av}\ge2.
\end{equation}
Indeed, if $d_{\rm av}<1/20$, then a standard Markov's inequality argument gives $|B_0^{(k)}|>19n_k/20$, contradicting
$|B_0^{(k)}|<\delta_k n_k$ for all large $k$.

Let $\varepsilon_k\to0$ be the diminishing sequence for the \lsfci{} condtion for $M^{(k)}$.
Set
\[
\widetilde n_k
:=
\min\left\{
\left\lceil c_{\tref{prop:produce-c4-good-subgraph}}\varepsilon_k^{-1} k\right\rceil,
n_k
\right\},
\qquad
\theta_k:=\frac{\widetilde n_k}{k}.
\]
Since $n_k/k\to\infty$ and $\varepsilon_k\to0$, we have
$\theta_k\to\infty$.
Since $r(k)\le k$, this also gives $\widetilde n_k\ge r(k)$ for all
sufficiently large $k$.
Set
\begin{equation}\label{eq: Lk defin}
L_k:=\min\left\{k,\frac{n_k}{k},\frac1{\varepsilon_k}\right\}.
\end{equation}
Then $L_k\to\infty$.

Use the two-stage exposure from Remark~\ref{rem:bridge-two-stage-exposure}, with
$\widetilde n_k=\theta_k k$. The final set $\Omega_k$ still has the
original uniform law on $\binom{J^{(k)}}{r(k)}$. Applying
Lemma~\ref{lem:weighted-subgraph-after-trimming} to the first stage, with probability
$1-o(1)$ there is a set
$J_{\rm sub}^{(k)}\subseteq\widetilde J^{(k)}$ such that
\[
|J_{\rm sub}^{(k)}|\ge \frac23\widetilde n_k,
\]
the graph $G^{(k)}[J_{\rm sub}^{(k)}]$ is $C_4$-free, every right support
degree in this subgraph lies between $2$ and $40d_{\rm av}$, and every nonzero weight in
this subgraph has absolute value at most $10d_{\rm av}$.
Writing
\[
Z:=|\Omega_k\cap J_{\rm sub}^{(k)}|,
\]
we have, that conditionally on the subgraph, $Z$ is a hypergeometric random variable with mean
at least $2r(k)/3$. Hoeffding's comparison theorem \cite[Theorem~4]{Hoeffding1963}
and the binomial Chernoff bound give $$Z\in[r(k)/2,r(k)]$$ with probability
$1-o(1)$. Conditionally on the subgraph and on this value of $Z$, the set
$\Omega_k\cap J_{\rm sub}^{(k)}$ is a uniform $Z$-subset of
$J_{\rm sub}^{(k)}$. It remains to show that such a uniform subset contains a
large \wcmrt{m_k}.

\smallskip
\noindent\emph{Finite-subgraph parameters.}
Let us fix any such realization $J_{\rm sub}^{(k)}$ and
$Z = r_{\rm sub} \in [r(k)/2,r(k)]$. The conditional law of $\Omega_k\cap J_{\rm sub}^{(k)}$ is a subset of size $r_{\rm sub}$ chosen uniformly from $J_{\rm sub}^{(k)}$. We will apply Theorem~\ref{thm:c4-finite-subgraph} to this conditional law, with the subgraph $G^{(k)}[J_{\rm sub}^{(k)}]$ and sample size $r_{\rm sub}$.

Since $J_{\rm sub}^{(k)}$ has cardinality
at least $\frac{2}{3}\theta_k k$, the ratio $\theta_{\rm sub}:=\frac{|J_{\rm sub}^{(k)}|}{k}$ satisfies
\[
\frac{2}{3}\theta_k
\le \theta_{\rm sub}\le \theta_k.
\]
Set
\[
\Lambda_{\rm sub}:=\min\left(r_{\rm sub}/2,\theta_{\rm sub}\right).
\]
The corresponding parameter $K_\star$ in Theorem~\ref{thm:c4-finite-subgraph} is the largest integer $K$ such that
\[
K\log K\le \Lambda_{\rm sub}.
\]
Moreover,
\[
\Lambda_{\rm sub}
\asymp \min\left(r(k)/2,\theta_k\right)
\asymp L_k,
\]
where $\asymp$ denotes comparability up to a constant factor depending only on
$\rho$ and $c_{\tref{prop:produce-c4-good-subgraph}}$.
Consequently,
\[
\log K_\star\asymp\log L_k \rightarrow \infty.
\]

Set $d_\star:=\lfloor40d_{\rm av}\rfloor$. By
\eqref{eq:average-degree-trichotomy-40d}, $d_\star\ge2$, and every right
degree in $G^{(k)}[J_{\rm sub}^{(k)}]$ is at most $d_\star$. From
Remark~\ref{rem:c4-finite-admissible-scale}, applied with $d=d_\star$ and
$\rho_{\rm sub}=r_{\rm sub}/k$, the finite theorem's parameter $m_\star$ is
uniformly bounded below by a quantity comparable to
\[
\frac{\log L_k}{d_{\rm av}^3\log (\tfrac{d_{\rm av}}{\rho}\log L_k)}
\to\infty,
\]
because $\rho_{\rm sub}\ge\rho/2$, and
$\log K_\star\asymp\log L_k$. Hence we may choose a deterministic integer
sequence $m_k\to\infty$ such that, for all relevant finite-subgraph
applications,
\[
m_k\le m_\star(r_{\rm sub},J_{\rm sub}^{(k)}).
\]
Thus, applying Theorem~\ref{thm:c4-finite-subgraph} with $m_k$, the  conditional probability that a uniform $r_{\rm sub}$-subset of $J_{\rm sub}^{(k)}$
contains a \wcmrt{m_k} is at least
\[
1 - \exp(-\sqrt{K_\star})-\frac{32d_\star}{\log K_\star}
= 1-o(1),
\]
for all relevant $r_{\rm sub}$ and subgraph realizations.

Combining this with the conditional distribution of $\Omega_k\cap J_{\rm sub}^{(k)}$, we obtain that the third case occurs with probability $1-o(1)$ in the two-stage construction.
In particular, we conclude that with high probability $\Omega_k$ contains a \wcmrt{m_k} with right support degrees at most $40d_{\rm av}$ and incident nonzero weights bounded by $10d_{\rm av}$.
Since the two-stage construction gives the original uniform law of $\Omega_k$,
the complementary case is proved. Combining the three cases proves the
proposition.

\end{proof}

\section{Proof of Theorem \ref{thm:weighted-local-sparse}}
\label{sec:matrix-theorem}

This section combines the graph cases with the matrix test-vector step.

\begin{proof}[Proof of Theorem~\ref{thm:weighted-local-sparse}]
Denote by $d_{\rm av}$ the uniform upper bound on the average degree.
We can also assume without loss of generality that the
average magnitude of non-zero entries is at most one.
If the support of $M^{(k)}$ is empty along a further subsequence, then
$M^{(k)}_{\Omega_k}$ is the zero matrix and the conclusion is immediate. Thus
we may also assume the support is nonempty.

Let $G^{(k)}=([k],[n_k],M^{(k)})$ be the weighted bipartite graph associated with $M^{(k)}$.

The average column-sparsity bound is exactly the average right support-degree bound for $G^{(k)}$, the weight hypothesis is exactly the average-magnitude hypothesis in Proposition~\ref{prop:average-degree-bridge-alternatives}, and the \lsdoubleoverlap{} is exactly the \lsfci{} condition for its support graph. By Proposition~\ref{prop:average-degree-bridge-alternatives}, there is a sequence $m_k\to\infty$ 
such that, with probability $1-o(1)$, one of the three trichotomy cases occurs for $G^{(k)}[\Omega_k]$.

If the sample contains an isolated right vertex, then $M^{(k)}_{\Omega_k}$ has a zero column, so its least singular value is $0$.

If the sample contains distinct degree-one right vertices $j,j'$ with the same unique left neighbor $i$, then the two sampled columns are both nonzero multiples of $e_i$. Hence they are linearly dependent, so the sampled matrix has a nontrivial kernel and again has least singular value $0$.

It remains to consider the tree case. The \wcmrt{m_k} has right support degrees at most $40d_{\rm av}$ and incident nonzero weights bounded by $10d_{\rm av}$. Thus this tree has $A_{\rm out}\le10d_{\rm av}$, and its third left layer has size at most $(40d_{\rm av})m_k(40d_{\rm av}-1)$. Use the test vector from Proposition~\ref{prop:weighted-tree-small-smin} on the selected tree columns and set all other sampled coordinates to zero. Extra sampled columns outside the tree have coefficient zero. Therefore
\[
s_{\min}\big(M^{(k)}_{\Omega_k}\big)
\le
10d_{\rm av}\sqrt{\frac{(40d_{\rm av})(40d_{\rm av}-1)}{m_k}}
\le
\frac{400d_{\rm av}^2}{\sqrt{m_k}}
=o(1)
\]
with probability $1-o(1)$.

\end{proof}

\begin{remark}[Slowly growing average degree]
\label{rem:slowly-growing-average-degree}
For readability, the main theorems are stated for fixed average degree
$d_{\rm av}$. However, the same proof gives
Proposition~\ref{prop:average-degree-bridge-alternatives} and
Theorem~\ref{thm:weighted-local-sparse} with $d_{\rm av}$ replaced by a sequence
of non-negative integers $d_k$, provided that
\[
d_k^4 
\ll
m_k
\asymp
\frac{\log L_k}{d_k^3\log( \tfrac{d_k}{\rho})\log L_k},
\]
where we need $d_k^4 = o(m_k)$ to ensure that the test vector from Proposition~\ref{prop:weighted-tree-small-smin} gives a vanishing upper bound on the least singular value. In particular, this works when $d_k \le \log^c L_k$ for some small constant $c>0$, where $L_k$ is defined in \eqref{eq: Lk defin}.
\end{remark}

\section{SparseStack and OSI}
\label{sec:sparsestack}

We now apply Theorem~\ref{thm:weighted-local-sparse}, together with
Remark~\ref{rem:slowly-growing-average-degree} for growing sparsity, to the
SparseStack construction from Definition~\ref{def:sparsestack} and the OSI property from Definition~\ref{def:osi}. The
orientation is different from the matrix theorem: SparseStack is an
$n\times k$ matrix $S$, while the theorem applies to $M:=S^\top\in\mathbb R^{k\times n}$.
Thus rows of $S$ become columns of $M$. Restricting $S^\top$ to a fixed set of
$r$ coordinates in $\mathbb R^n$ is the same as selecting $r$ columns of $M$.
The SparseStack rows are exchangeable, so a deterministic coordinate set can
be considered in place of the uniform column sample used in the matrix theorem. The next proposition immediately implies corollary from the introduction. We deliberately consider a more general setting allowing for a growing sparsity parameter:

\begin{prop}[SparseStack satisfies the local sparse double-overlap condition]
\label{prop:sparsestack-double-overlap}

Let $\zeta_k\ge2$ be an integer sequence such that $k$ is
an integer multiple of $\zeta_k$, set
$b(k):=k/\zeta_k$, and let $n_k$ satisfy $n_k/k\to\infty$.
Further, assume that $\zeta_k^4/k\to 0$, and
let $a_k$ be any sequence such that
\[
\varepsilon_k:=a_k\frac{\zeta_k^4}{k}\to0;\quad a_k\to\infty.
\]
Let $S^{(k)}\in\mathbb R^{n_k\times k}$ be a SparseStack matrix with row
sparsity $\zeta_k$ and block size $b(k)$, and define
$M^{(k)}:=(S^{(k)})^\top\in\mathbb R^{k\times n_k}$.
Then the matrix $M^{(k)}$ has
column sparsity $\zeta_k$ and, with probability $1-o(1)$, satisfies
the \lsdoubleoverlap{} with parameter $\varepsilon_k$.

\end{prop}

\begin{proof}

Each column of $M^{(k)}$ is the transpose of a row of $S^{(k)}$, so it has
exactly one nonzero in each of the $\zeta_k$ blocks and hence exactly
$\zeta_k$ nonzero entries.

Fix distinct column indices $j,j'\in[n_k]$. The overlap $|\operatorname{supp}M^{(k)}_{\cdot j}\cap \operatorname{supp}M^{(k)}_{\cdot j'}|$ counts the number of blocks in which the corresponding SparseStack rows choose the same coordinate. Therefore
\[
\mathbb P\big(
|\operatorname{supp}M^{(k)}_{\cdot j}\cap \operatorname{supp}M^{(k)}_{\cdot j'}|\ge2
\big)
\le
\binom{\zeta_k}{2}\frac1{b(k)^2}
\le
\frac{\zeta_k^4}{2k^2}.
\]

For each $j\in[n_k]$, let $B_j$ be the number of $j'\ne j$ whose support
intersects the support of $j$ in at least two coordinates. Conditional on the
$j$-th SparseStack row, the rows indexed by $j'\ne j$ remain independent, so
the corresponding indicators are independent and have probability at most
$\zeta_k^4/(2k^2)$. Hence
\[
\mathbb E B_j\le \frac{\zeta_k^4n_k}{2k^2}.
\]
Therefore
\[
\frac{\mathbb E B_j}{\varepsilon_k n_k/k}
\le
\frac1{2a_k}
=o(1)
\]
uniformly in $j$. By Markov's inequality,
$\mathbb P(B_j>\varepsilon_k n_k/k)=o(1)$ uniformly in $j$.

Let $Y_k$ be the number of indices $j$ for which $B_j>\varepsilon_k n_k/k$. Then
$\mathbb E Y_k=o(n_k)$, so $Y_k=o(n_k)$ with probability $1-o(1)$. On this event, the set
\[
\mathcal G_k:=
\left\{
j\in[n_k]:
B_j\le\varepsilon_k\frac{n_k}k
\right\}
\]
has cardinality $(1-o(1))n_k$, and every $j\in\mathcal G_k$ has at most $\varepsilon_k n_k/k$ \doubleoverlappartners{}. This is exactly the \lsdoubleoverlap{}.

\end{proof}

Observe that, for any constant sparsity parameter $\zeta:=\zeta_k$,
and in the regime $n_k/k\to\infty$, the above proposition
shows that the transpose of the SparseStack matrix model satisfies
all the assumptions of Theorem~\ref{thm:weighted-local-sparse} with high probability, and Corollary~\ref{cor:sparsestack-no-osi} follows.
Moreover, combined with a quantitative version of the main theorem
(see Remark~\ref{rem:slowly-growing-average-degree}),
it gives the following result where the sparsity parameter
is allowed to grow with $k$ as a small power of $\log\min(n_k/k,k))$, assuming 
polynomial-order aspect ratio:

\begin{prop}[SparseStack with a polylogarithmic row sparsity]
\label{prop:sparsestack-fails-osi}
There is a universal constant $c>0$ with the following property.
Let $\zeta_k\ge2$ be an integer sequence with $\zeta_k\mid k$, set
$b(k):=k/\zeta_k$, and let $r=r(k)$ and $n_k$ satisfy
\[
\rho k\le r(k)\le k
\quad\text{for some fixed }\rho\in(0,1],
\qquad
\frac{n_k}{r(k)}\to\infty.
\]
Set
\[
s_k:=\min\left\{k,\frac{n_k}{k}\right\},
\]
and assume
\[
\zeta_k=O\big(\log s_k\big)^c.
\]
For each such $k$, let $S^{(k)}\in\mathbb R^{n_k\times k}$ be a SparseStack
matrix with row sparsity $\zeta_k$ and block size $b(k)$. Let
$T_k\subseteq[n_k]$ be any deterministic subset with $|T_k|=r(k)$.
Then
\[
s_{\min}\left(((S^{(k)})^\top)_{T_k}\right)=o(1)
\]
with probability $1-o(1)$. In particular, 
the SparseStack with the given parameters
is not $(r(k),1/2)$-OSI.

\end{prop}

We leave proof details to the interested reader.

\section{Open Problems}

The main structural assumption in this note controls double overlaps between typical pairs of columns. The most natural question is whether this assumption is actually needed.

\begin{conjecture}[Unrestricted deterministic sparse sketches]
Fix an integer $d \ge 2$ and a constant $C>0$. Let
$r(k),n_k\in\N$
satisfy
\[
    \frac{k}{r(k)}\to C,
    \qquad
    \frac{n_k}{r(k)}\to\infty.
\]
For each $k$, let
\[
    M^{(k)}\in\{-1,0,1\}^{k\times n_k}
\]
be a deterministic matrix whose every column has exactly $d$ nonzero entries. If
$\Omega_k\subset[n_k]$
is chosen uniformly from $\binom{[n_k]}{r(k)}$,
then
\[
    \smin\bigl(M^{(k)}_{\Omega_k}\bigr)=o(1)
\]
with probability $1-o(1)$.
\end{conjecture}

Our proof does not address the above conjecture
since it uses a locally tree-like cancellation pattern built from typical columns having very few \doubleoverlappartners{}. 
However, we expect that double overlaps cannot create an obstruction to vanishing $\smin$.

Another question is to obtain optimal interdependence between the sparsity parameter (number of non-zero elements per column) and the magnitude of the smallest singular value. The question is open even in the setting of random matrices:
\begin{problem}[Optimal OSI sparsity for combinatorial random matrices]
Assume that $n_k\geq k^{1+c}$, for a small positive constant $c>0$.
For each $k$ let $M^{(k)}$
be a chosen uniformly among
$\{-1/\sqrt{d_k},0,+1/\sqrt{d_k}\}$--matrices of dimension $k\times n_k$, having $d_k$ non-zero elements in every column.
What are optimal necessary and sufficient conditions on 
the growth of sequence $d_k$
to guarantee that uniform random $k\times k/100$ submatrix
of $M^{(k)}$ has $\smin$ of order $\Omega(1)$ with high probability?
\end{problem}

In regard to the last problem, it is natural to expect 
that there is a critical power $\alpha_*>0$ such that whenever
$d_k=O(\log^{\alpha_*-\varepsilon})$ for some $\varepsilon>0$,
the uniform random $k\times k/100$ submatrix of $M^{(k)}$ is poorly invertible,
whereas for $d_k=\Omega(\log^{\alpha_*+\varepsilon})$
it is well invertible (has constant $\smin$) with high probability.
We further expect that the parameter $\alpha_*$ is the same
for the above model of randomness and for SparseStack$^\top$.
Determining the exact value of $\alpha_*$
is an intriguing open question.

\end{document}